\documentclass[reqno]{amsart}
\usepackage{hyperref,xcolor,amssymb}

\DeclareMathOperator*{\esssup}{ess\,sup}

\DeclareMathOperator*{\essinf}{ess\,inf}

\newcommand{\osc}{\mathrm{osc}}

\def\Xint#1{\mathchoice
{\XXint\displaystyle\textstyle{#1}}%
{\XXint\textstyle\scriptstyle{#1}}%
{\XXint\scriptstyle\scriptscriptstyle{#1}}%
{\XXint\scriptscriptstyle\scriptscriptstyle{#1}}%
\!\int}
\def\XXint#1#2#3{{\setbox0=\hbox{$#1{#2#3}{\int}$ }
\vcenter{\hbox{$#2#3$ }}\kern-.6\wd0}}

\def\dashint{\Xint-}

\begin{document}
\title[Quasiminimizers of a $(p,q)$-Dirichlet integral]
{Regularity properties for quasiminimizers of a $(p,q)$-Dirichlet integral}

\author[A. Nastasi, C. Pacchiano Camacho]
{Antonella Nastasi, Cintia Pacchiano Camacho}

\address{Antonella Nastasi \newline
University of Palermo, Department of Mathematics and Computer Science, Via Archirafi 34, 90123, Palermo, Italy}
\email{antonella.nastasi@unipa.it, antonellanastasi.math@gmail.com}

\address{Cintia Pacchiano Camacho\newline
	Aalto University, Department of Mathematics and Systems Analysis, Espoo, Finland}
\email{cintia.pacchiano@aalto.fi}

\subjclass[2010]{primary 31E05; secondary 30L99, 46E35}
\keywords{$(p,q)$-Laplace operator; measure metric spaces; minimal $p$-weak upper gradient; minimizer.}

\begin{abstract}
Using a variational approach we study interior regularity for quasiminimizers of a $(p,q)$-Dirichlet integral, as well as regularity results up to the boundary, in the setting of a metric space equipped with a doubling measure and supporting a Poincar\'{e} inequality. 
For the interior regularity, we use De Giorgi type conditions to show that quasiminimizers are locally H\"{o}lder continuous and they satisfy Harnack inequality, the strong maximum principle and Liouville's Theorem. Furthermore, we give a pointwise estimate near a boundary point, as well as a sufficient condition for H\"older continuity and a Wiener type regularity condition for continuity up to the boundary.
Finally, we consider $(p,q)$-minimizers  and we give an estimate for their oscillation at boundary points.
\end{abstract}

\maketitle
\numberwithin{equation}{section}
\newtheorem{theorem}{Theorem}[section]
\newtheorem{lemma}[theorem]{Lemma}
\newtheorem{proposition}[theorem]{Proposition}
\newtheorem{remark}[theorem]{Remark}
\newtheorem{definition}[theorem]{Definition}
\newtheorem{corollary}[theorem]{Corollary}
\allowdisplaybreaks

\section{Introduction}
The aim of this paper is to study quasiminimizers of the following anisotropic energy $(p,q)$-Dirichlet integral
\begin{equation}\label{J}
\int_{\Omega}  a g_{u}^p d \mu + \int_{\Omega}b g_{u}^q d \mu,
\end{equation}
in metric measure spaces, with $g_u$ the minimal $q$-weak upper gradient of $u$. Here, $\Omega\subset X$ is an open bounded set, where $(X, d, \mu)$ is a complete metric measure space with metric $d$ and a doubling Borel regular measure $\mu$, supporting a weak $(1,p)$-Poincar\'{e} inequality for $1<p<q$. We consider some coefficient functions $a$ and $b$ to be measurable and satisfying $0<\alpha\leq a,b\leq \beta$, for some positive constants $\alpha$, $\beta$.

One of the advantages of working in the general setting is that variational methods, such as those based on De Giorgi classes \cite{D}, are still available (see \cite{GG}). In \cite{GG2}, Giaquinta and Giusti introduced the notion of quasiminimizers in $\mathbb{R}^{N}$. They proved several of their fundamental properties such as local H\"older continuity and the strong maximum principle. We also mention the paper by DiBenedetto and Trudinger \cite{DiBT}, where they proved the Harnack inequality for quasiminimizers. The study of calculus on metric measure spaces has attracted a lot of attention in recent years, it has been proven that Sobolev spaces can be defined without the notion of partial derivatives, (see \cite{BB, BBS, C, HKS, N, S}). This theory can be applied in several areas of analysis, for example, calculus on Riemannian manifolds, subelliptic operators associated with vector fields, potential theory on graphs and weighted Sobolev spaces. 

Since in  a metric measure space it is possible to define (quasi)minimizers of Dirichlet integrals, this appropach is particularly useful. Local properties of quasiminimizers of the $p$-energy integral on metric spaces were studied by Kinnunen-Shanmugalingam in \cite{KS}. More precisely, they discussed regularity properties. They used the De Giorgi method to prove that, if the metric measure space is equipped with a doubling measure and it supports a Poincar\'{e} inequality, quasiminimizers of the $p$-energy functional are locally H\"{o}lder continuous, they satisfy the Harnack inequality and the maximum principle.

For our work, we study the $(p,q)$-Dirichlet integral \eqref{J} on a metric space equipped with a doubling measure and supporting a Poincar\'{e} inequality as in \cite{KS}, but, as we have already anticipated, the new feature is that we include both $p$-Laplace and $q$-Laplace operators, involving also some measurable coefficient functions $a$ and $b$, only assuming they are bounded away from $0$ and $\infty$. This condition over the coefficients is essential for our approach, however it is an open question if it could be relaxed. By adapting the approach used in \cite{KS}, we establish the local boundedness for quasiminimizers of the convex integral \eqref{J}. More precisely, we show that quasiminimizers satisfy a De Giorgi type inequality \cite{D} and some boundedness properties, thus extending the corresponding results in \cite{KS}. In addition, we study regularity up to the boundary. We give a comparison principle for $(p,q)$-minimizers. As a consequence, we get an estimate for the oscillation of $(p,q)$-minimizers at boundary points. This result extends the work in \cite{BMS} since this only concerns $p$-minimizers. We note that some qualitative results would follow from the theory of quasiminimizers by Kinnunen and Shanmugalingam \cite{KS}, however the novelty here is that we are interested not only in qualitative results but also in quantitative estimates with an explicit analysis on the dependencies of the constants. In our approach the constants will depend on the upper and lower bounds $\alpha$ and $\beta$, as well as on the $(p,q)$-quasiminimizer constant $K$.

Recently, many authors have explored different generalizations of classical elliptic and parabolic partial differential equations, such as the nonlinear $p$ and $(p,q)$-Laplace equations, (see, for example \cite{BDM,BDMS2,CM1,CM2,Maz}). However, there are still new and interesting open mathematical questions in the setting of anisotropic nonlinear elliptic and parabolic partial differential equations driven by $(p,q)$-Laplace operators.
In general, anisotropic partial differential equations have received a large interest due to their applications to double variational energies and anisotropic energies in integral form. 
There exists a rich literature concerning regularity results for solutions to partial differential equations, both elliptic and parabolic, under $p$ and $(p,q)$-growth conditions in the Euclidean setting. For more information regarding other $(p,q)$-growth cases, we refer to the papers of Baroni, Colombo and Mingione \cite{BCM}, Cupini, Marcellini and Mascolo \cite{CMM1, CMM2}, De Filippis and Oh \cite{DO}, Hadzhy, Skrypnik and Voitovych \cite{HSV}, Mingione and R\v{a}dulescu \cite{MR}, Marcellini \cite{Ma1, Ma2, Ma3}. In our study, we focus on the anisotropic energy integral as presented by Marcellini in \cite{Ma0}, but the new feature is that we work in metric measure spaces. We have considered this more general setting to prove that the special case of $(p,q)$-growth condition \eqref{J} can be treated also in a very general context, thus obtaining several relevant properties for (quasi)minimizers even in a metric framework.

Furthermore, there are also some existing regularity results concerning the boundary behaviour for (quasi)minimizers, both in the Euclidean (see \cite{DiG,T,Z}) and in the metric setting (see \cite{B, BJ}). More specifically, Ziemer \cite{Z} proved a Wiener type condition for the continuity of a quasiminimum at a boundary point of a bounded open subset of $\mathbb{R}^N$. On the other hand, Bj\"{o}rn \cite{B} extended these results to the general metric setting and also gave sufficient condition for H\"{o}lder continuity. Another important contribution concerning boundary behaviour can be found in \cite{BMS}. In this paper, Bj\"{o}rn, MacManus and Shanmugalingam obtained an estimate for the oscillation of $p$-harmonic functions and $p$-energy minimizers near a boundary point. However, the study of boundary behavior for the $(p,q)$-problems can be considered mostly still open, at least in its full generality.

The present work is divided as follows.  Section \ref{Sec2} constitutes the mathematical background. There we give definitions and useful notation needed throughout the paper. In Section \ref{quasimin} we introduce $(p,q)$-(quasi)minimizers and prove uniqueness properties for $(p,q)$-minimizers. Then, in Section \ref{Sec3}, we obtain that $(p,q)$-quasiminimizers satisfy a De Giorgi type inequality. 
In the following two sections local H\"older continuity (Section \ref{Sec4}) and Harnack's inequality are studied (Section \ref{Sec5}). As a consequence, we obtain a strong maximum principle and prove Liouville's Theorem.
Furthermore, in Sections \ref{Sec6} and \ref{Sec7}, we focus on regularity results up to the boundary. In particular, in Section \ref{Sec6}, we study H\"older continuity at a boundary point and give a Wiener type regularity condition for continuity at the boundary for $(p,q)$-quasiminimizers. Finally, in Section \ref{Sec7}, we prove an estimate which gives us control over the oscillation of $(p,q)$-minimizers at boundary points.
\section*{Acknowledgements}
The authors wish to thank Professor Juha Kinnunen for supporting us in our research and for all the enlightening discussions and advice. Special thanks go to Professor Paolo Marcellini for the panoramical view he gently provided on the existing literature and open questions on the subject. The authors are grateful to the referees for their careful reading and the useful comments.
The second author was supported by a doctoral training grant for 2021 from the V\"ais\"al\"a Fund.

\section{Mathematical background}\label{Sec2}

Let $(X, d, \mu)$ be a complete metric measure space, where $\mu$ is a Borel regular measure with $\mu(\Omega)>0$ for every $\Omega\subset X$ non empty open set and $\mu(B)<+\infty$ for every $B\subset X$ bounded set. Let $B(y,\rho)\subset X$ be a ball with the center $y \in X$ and the radius $\rho>0$. 
 We denote $$u_S= \dfrac{1}{\mu(S)}\int_{S} u \, d\mu=\dashint_{S}u\, d\mu,$$where $S \subset X$ is a measurable set of finite positive measure and $u: S \to \mathbb{R}$ is a measurable function. Throughout this paper, we will indicate with $C$ all positive constants, even if they assume different values, unless otherwise specified.
 
\begin{definition}[\cite{BB}, Section 3.1]
	A measure $\mu$ on $X$ is said to be doubling if there exists a constant $C_d \geq 1$, called the doubling constant, such that 
	\begin{equation}\label{doubling}
	0<\mu(B(y,2\rho))\leq C_d \, \mu(B(y,\rho))< +\infty,
	\end{equation} for all $y \in X$ and $\rho>0$.
\end{definition}
\begin{lemma}[\cite{BB}, Lemma 3.3]\label{lemm3.3}
	Let $(X, d, \mu)$ be a metric measure space with $\mu$ doubling. Then there is $Q>0$ such that
	\begin{equation}\label{s}
	\dfrac{\mu(B(y,\rho))}{\mu(B(x, R))}\geq C\left(\dfrac{\rho}{R}\right)^Q
	\end{equation}
	for all $\rho\in ]0, R]$, $x \in {\Omega}$, $y \in B(x, R)$, where constants $Q$ and $C$ depend only on $C_d$.
\end{lemma}
Now, we introduce the concept of upper gradient following the notations in the book by  Bj\"{o}rn and Bj\"{o}rn \cite{BB}.
\begin{definition}[\cite{BB}, Definition 1.13]
A non negative Borel measurable function $g$ is said to be an upper gradient of function $u: X \to [-\infty, +\infty]$ if, for all compact rectifiable arc lenght parametrized paths $\gamma$ connecting $x$ and $y$, we have
\begin{equation}\label{ug}
|u(x)-u(y)|\leq \int_{\gamma}g\, ds
\end{equation}
 whenever $u(x)$ and $u(y)$ are both finite and $\int_{\gamma}g \, ds= +\infty$ otherwise.
\end{definition}
The notion of upper gradient has been introduced to overcome the lack of a differentiable structure in metric measure spaces. 
From this last definition, we note that  we don't have uniqueness of the upper gradient.
Indeed, if $g$ is an upper gradient of function $u$ and $\phi$ is any non negative Borel measurable function, then $g+\phi$ is still an upper gradient of $u$. 
As a consequence, we need the concept of $q$-weak upper gradient. If $g$ fulfills \eqref{ug} for $q$-almost all paths, meaning that the family of non constant
paths for which \eqref{ug} fails is of zero $q$-modulus (see \cite{BB}, Definition 1.33), then $g$ is called $q$-weak upper gradient of $u$. 

The following theorem states the existence of a minimal element for the family of $q$-weak upper gradients of $u$, which is $\mu$-a.e. uniquely determinated.
\begin{theorem}[\cite{BB}, Theorem 2.5]\label{mpwug}
Let $q \in ]1, +\infty[$. Suppose that $u\in L^q(X)$ has an $L^q(X)$ integrable $q$-weak upper gradient. Then there exists a $q$-weak upper gradient, denoted with $g_u$, such that $g_u\leq g$ $\mu$-a.e. in $X$, for each $q$-weak upper gradient $g$ of $u$. This $g_u$ is called the minimal $q$-weak upper gradient of $u$.
\end{theorem}

\noindent We remark that $g_u$ of Theorem \ref{mpwug} is a generalization of the Euclidean modulus of the gradient of $u$ to the metric case. In general, the upper gradients of a function do not necessarily give us a control over it. In order to gain such control one standard hypothesis when working in the metric setting is to assume that the space supports a Poincar\'{e} inequality.

\begin{definition}[\cite{BB}, Definition 4.1]\label{PI}
	Let $s \in [1, +\infty[$. 
	A metric measure space $X$ supports a weak $(1, s)$-Poincar\'{e} inequality if there exist $C_{PI}$ and a dilation factor $\lambda \geq 1 $ such that 
	\begin{equation*}
	\dashint_{B(y,r)} |u-u_{B(y,r)}|d\mu\leq C_{PI} r \left(\dashint_{B(y,\lambda r)}g_u^s \, d\mu\right)^{\frac{1}{s}}
	\end{equation*}
	for all balls $B(y,r) \subset X$ and for all $u \in L^1_{loc}(X)$.
\end{definition}
The following results show some self improving properties of the weak $(1,s)$-Poincar\'{e} inequality.
\begin{theorem}[\cite{BB}, Theorem 4.21]\label{sstars}
Assume that $X$ supports a weak $(1,s)$-Poincar\'{e} inequality and that $Q$ in (\ref{s}) satisfies $Q>s$. Then $X$ supports a weak $(s^*,s)$-Poincar\'{e} inequality with $s^*=\frac{Qs}{Q-s}$. More precisely, there are constants $C>0$ and a dilation factor $\lambda'>1$ such that
\begin{equation}\label{(2.8)}
	\left(\dashint_{B(y,r)} |u-u_{B(y,r)}|^{s^*} \,d\mu\right)^{\frac{1}{s^{*}}}\leq C r \left(\dashint_{B(y,\lambda' r)}g_u^s \, d\mu\right)^{\frac{1}{s}},
\end{equation}
for all balls $B(y, r)\subset X$ and all integrable functions $u$ in $B(y, r)$. The dilation factor $\lambda'$ depends on $\lambda$ from Definition \ref{PI}.
\end{theorem}

\begin{corollary}[\cite{BB}, Corollary 4.26]\label{coropoinca}
If $X$ supports a weak $(1,s)$-Poincar\'{e} inequality and $Q$ in (\ref{s}) satisfies $Q\leq s$, then $X$ supports a weak $(t,s)$-Poincar\'{e} inequality for all $1\leq t<\infty$.
\end{corollary}

\begin{remark}\label{rem1poin} By the H\"older inequality we see that a weak $(s^*,s)$-Poincar\'{e} inequality implies the same inequality for smaller values of $s^*$. Meaning that $X$ will then support a weak $(t,s)$-Poincar\'{e} inequality for all $1<t<s^*$.
\end{remark}

\begin{remark}\label{rem2poin} We note that the exponent $Q$ in (\ref{s}) is not uniquely determined, in particular, since $\rho< R$, we can always make $Q$ larger. Thus, the assumption $Q>s$ in Theorem \ref{sstars} can always be fulfilled. 
\end{remark}

\begin{theorem}[\cite{KZ}]\label{spoincare}
Let $(X, d, \mu)$ be a complete metric measure space with $\mu$ Borel and doubling, supporting a weak $(1,p)$-Poincar\'{e} inequality for $p>1$, then there exists $\epsilon>0$ such that $X$ supports a weak $(1, s)$-Poincar\'{e} inequality for every $s>p-\epsilon$. 
\end{theorem}

 Before introducing the  $(p,q)$-Dirichlet boundary value problem, we need a suitable working space. The Newtonian space $N^{1,q}(X)$ is defined by 
  \begin{equation}\label{space}N^{1,q}(X)=V^{1,q}(X)\cap L^q (X), \quad q\in [1,+\infty],
  \end{equation} where
 $ V^{1,q}(X) = \{u: u \ \mbox {is measurable and }  g_u\in L^q (X) \}.$
 We consider $N^{1,q}(X)$ equipped with the norm $$\|u\|_{N^{1,q}(X)}=\|g_u\|_{L^q (X)} + \|u\|_{L^q (X)}.$$ 
We can naturally consider $\Omega\subset X$ non empty open set as a metric space in its own right (with the restrictions of $d$ and $\mu$). The Newtonian space $N^{1,q}(\Omega)$ is then given by \eqref{space}.
The Newtonian space with zero boundary values is defined as
$$N^{1,q}_0(\Omega)=\{u_{| \Omega}: u \in N^{1,q}(X), u=0 \quad \mbox{in } X \setminus \Omega \}.$$
That is, a function belongs to $N^{1,q}_0(\Omega)$ if and only if its zero extension to $X \setminus \Omega$ belongs to $N^{1,q}(X)$. We shall therefore always assume that $\mu(X \setminus \Omega)>0$.

Let $L^q_{loc} (\Omega)$ be the space of all measurable functions that are $q$-integrable on bounded subsets of $X$.
The space $N^{1,q}_{loc}(\Omega)$ is defined by $$N^{1,q}_{loc}(\Omega)=V^{1,q}_{loc}(\Omega)\cap L^q_{loc} (\Omega), \quad q\in [1,+\infty],$$ where
	$ V^{1,q}_{loc}(\Omega) = \{u: u \ \mbox {is measurable and }  g_u\in L^q_{loc} (\Omega) \}.$
	
The following lemma implies a Sobolev inequality for Sobolev functions with zero boundary values (see Corollary \ref{RemarkPoincare}).
\begin{lemma}[\cite{KS}, Lemma 2.1]\label{Lemma 2.1 KS}
 Let $X$ be a doubling metric measure space supporting a weak
$(1, s)$-Poincar\'{e} inequality for some $1< s < q$. Suppose that $u \in N^{1,q}(X)$ and let $A = \{x \in B(y,R): |u(x)| > 0\}$. If $\mu(A) \leq \gamma\mu(B(y,R))$ for some $\gamma$ with $0 < \gamma < 1$, then there is a constant $C>0$ so that
\begin{equation*}
\left(\dashint_{B(y,r)} |u|^t \, d\mu\right)^{\frac{1}{t}}\leq C r \left(\dashint_{B(y,\lambda' r)}g_u^s \, d\mu\right)^{\frac{1}{s}},
\end{equation*}
where $t $ is given by Remark \ref{rem1poin} and $\lambda'$ is as in \eqref{(2.8)}. The constant $C$ depends only on $\gamma$ and the constants $C$ and $\lambda'$ of \eqref{(2.8)}.
\end{lemma}
\begin{corollary}[see \cite{KS}]\label{RemarkPoincare}
Under the same hypotheses of the previous lemma, there exists $C>0$ so that for every ball $B(y, r)$ with $0 < r \leq\frac{{\rm diam}(X)}{3}$ and every $u \in N^{1,q}_0 (B(y, r))$ we have
\begin{equation}\label{2.6KS}
\left(\dashint_{B(y,r)} |u|^t \, d\mu\right)^{\frac{1}{t}}\leq C r \left(\dashint_{B(y,r)}g_u^s \, d\mu\right)^{\frac{1}{s}}
\end{equation}where $t $ is given by Remark \ref{rem1poin}. 
\end{corollary}

\begin{definition}[\cite{BB}, Definition 6.13] Let $E\subset X$ be a Borel set. We define the $q$-capacity of $E$ as
		\begin{equation*}
			{\rm C}_{q}(E)=\inf_{u}\left(\int_{X} |u|^{q} \, d\mu + \inf_{u}\int_{X} g_{u}^{q} \, d\mu\right),
		\end{equation*}
		where the infimum is taken over all $u\in N^{1,q}(X)$ such that $u=1$ on $E$.
	\end{definition}	
	We say that a property holds $q$-quasieverywhere ($q$-q.e) if the set of points
	for which it does not hold has q-capacity zero.
	
We report here the definition of variational capacity, that is involved in the pointwise estimate for $(p,q)$-quasiminimizers. 
\begin{definition}[\cite{BB}, Definition 6.13] Let $B(y, r)\subset X$ be a ball and $E\subset B(y, r)$. We define the variational capacity
	\begin{equation*}
		{\rm cap}_{q}(E,B(y, 2r))=\inf_{u}\int_{B(y, 2r)}g_{u}^{q} \, d\mu
	\end{equation*}
	where the infimum is taken over all $u\in N_{0}^{1,q}(B(y, 2r))$ such that $u\geq 1$ on $E$.
\end{definition}

\begin{remark}[\cite{BB}, Corollary A.9]\label{gradients}
	Let  $1<p<q$, $u \in N^{1,q}(X)$. If $(X, d, \mu)$  is a complete doubling $(1,p)$-Poincar\'{e} space, then the minimal $p$-weak upper gradient and the minimal $q$-weak upper gradient of $u$ coincide $\mu$-a.e.
\end{remark}

\textit{Throughout this paper, we consider a complete metric measure space $(X, d, \mu)$ with metric $d$ and a doubling Borel regular measure $\mu$. In virtue of Remark \ref{gradients}, we consider $q$-weak upper gradients rather than $p$-weak upper gradients. Moreover, we assume that $X$ supports a weak $(1,p)$-Poincar\'{e} inequality with $1<p<q<p^*$. From now on and without further notice, we fix $1<s<p$ for which $X$ also admits a weak  $(1,s)$-Poincar\'e  inequality. Such $s$ is given by Theorem \ref{spoincare} and will be used in various of our results.
We also consider a non empty open subset $\Omega\subset X$ such that $\mu(X \setminus \Omega)>0$.}
\section{$(p,q)$-quasiminizers}\label{quasimin}
In this note, we are interested in anisotropic energy integrals, which satisfy the so called $(p,q)$-growth condition for some exponents $1<p<q$.  Since we work in metric measure spaces and due to the methods we use, we treat it under sharp assumptions. That is, 
\begin{equation}\label{integral}
	\int_{\Omega} a g_{u}^p\, d \mu + \int_{\Omega} b g_{u}^q\, d \mu,
\end{equation}
for some bounded measurable functions $a,b:X \to \mathbb{R}$ with $0<\alpha\leq a,b \leq \beta$, for some positive constants $\alpha, \beta$. This is a relevant perturbation of the $p$-Dirichlet integral. For more information on this kind of problems we refer the reader to \cite{Ma0} and the references therein.

Now, we introduce the definition of $(p,q)$-quasiminimizers of integral \eqref{integral}.
\begin{definition}[$(p,q)$-quasiminimizer]\label{qm}
	A function $u\in N^{1,q}_{loc}({\Omega})$ is a $(p,q)$-quasi-\\minimizer on $\Omega$ if there exists $K\geq 1$ such that for every bounded open subset $\Omega'$ of $\Omega$ with $\overline{\Omega'}\subset \Omega$ and for all functions $v \in N^{1,q}({\Omega'})$ with $u-v\in  N^{1,q}_0({\Omega'})$ the inequality
\begin{align}\label{min}
\int_{\Omega'} a g_{u}^p\, d \mu + \int_{\Omega'} bg_{u}^q\, d \mu
\leq& K\left(\int_{\Omega'}a g_{v}^p\, d \mu + \int_{\Omega'} b g_{v}^q \,d \mu \right)
\end{align} 
holds,  where $g_{u}$, $g_{v}$ are the minimal $q$-weak upper gradients of $u$ and $v$ in ${\Omega}$, respectively.
If $K=1$, then $u$ is called $(p,q)$-minimizer on $\Omega$. Furthermore, a function $u\in N^{1,q}({\Omega})$ is a global $(p,q)$-minimizer on $\Omega$ if \eqref{min} is satisfied with $\Omega$ instead of $\Omega'$ and for all $v \in N^{1,q}_0({\Omega})$.
\end{definition}
From now on, to simplify notation we refer to global $(p,q)$-minimizers by just writing $(p,q)$-minimizers.

We obtain uniqueness of $(p,q)$-minimizers as a corollary of the next theorem, which states that their $p$-weak upper gradients coincide. This will also be used to prove a boundary regularity result (see Section \ref{Sec7}). 
\begin{theorem}\label{uniqueness}
	Let $u_1, u_2 \in N^{1,q}_{loc}({\Omega})$,with $u_1-u_2 \in  N^{1,q}_0({\Omega})$,  be  $(p,q)$-minimizers on $\Omega$. Then $g_{u_1}= g_{u_2}$ $\mu$-a.e. on $\Omega$.
\end{theorem}
\begin{proof}
	Let $\Omega'$ a bounded open subset of $\Omega$  with $\overline{\Omega'}\subset \Omega$. 
	By absurd, let $\mu(\{x \in {\Omega'} :g_{u_1}\neq g_{u_2}\}) >0$. Then we can choose $\delta>0$ such that $D_{\delta}= \{x \in {\Omega'} :|g_{u_1}-g_{u_2}|>\delta\}$ has positive measure. 
	We consider $$u=\dfrac{u_1+u_2}{2}.$$ From the definition of minimal $p$-weak upper gradient, we have that $$g_{u}\leq \dfrac{g_{u_1}+g_{u_2}}{2}.$$ The function $t\mapsto t^l$	is uniformly convex on $[0, +\infty[$, thus there exists $\epsilon= \delta^l l (2^{-1}-2^{-l})$, with $l=p,q$, such that $$\left(\frac{g_{u_1}+g_{u_2}}{2}\right)^l\leq \frac{g_{u_1}^l+g_{u_2}^l}{2}-\epsilon_l,$$
	where $|g_{u_1}-g_{u_2}|\geq\delta$. \\
	Let $\epsilon=\min\{\epsilon_1, \epsilon_2\}$.
	As a consequence, we get that 
	\begin{align}\label{uni}
		 \int_{\Omega'} a g_{u}^p d\mu+ \int_{\Omega'} b g_{u}^q d\mu
		\leq& \int_{D_{\delta}}a\left(\frac{g_{u_1}^p+g_{u_2}^p}{2}- \epsilon\right)d\mu+\int_{{\Omega'}\setminus D_{\delta}}a\left(\frac{g_{u_1}^p+g_{u_2}^p}{2}\right)d\mu\nonumber \\ 	&+\int_{D_{\delta}}b\left(\frac{g_{u_1}^q+g_{u_2}^q}{2}- \epsilon\right)d\mu+\int_{{\Omega'}\setminus D_{\delta}}b\left(\frac{g_{u_1}^q+g_{u_2}^q}{2}\right)d\mu\nonumber \\
		=&\frac{1}{2}\left(\int_{\Omega'}ag_{u_1}^p d\mu+\int_{\Omega'}bg_{u_1}^q d\mu\right) 
		\nonumber \\&+\frac{1}{2}\left(\int_{\Omega'}ag_{u_2}^p d\mu+\int_{\Omega'}b g_{u_2}^q d\mu\right) 
		-2\beta \epsilon \mu (D_{\delta})\nonumber \\ 
		=&	\int_{\Omega'}ag_{u_1}^p d\mu+\int_{\Omega'}bg_{u_1}^q d\mu
		-2\beta \epsilon \mu (D_{\delta}).
	\end{align}
	Since $u_1$ is a $(p,q)$-minimizer and using \eqref{uni}, we get
	\begin{align*}
		 \int_{\Omega'} ag_{u_1}^p d\mu+ \int_{\Omega'}b g_{u_1}^q d\mu
		\leq & \int_{\Omega'} ag_{u}^p d\mu+ \int_{\Omega'} bg_{u}^q d\mu\nonumber \\	
		\leq & \int_{\Omega'}a g_{u_1}^p d\mu+ \int_{\Omega'}b g_{u_1}^q d\mu -2\beta \epsilon \mu (D_{\delta}),
	\end{align*}
	that is absurd. Thus, we have proven that $g_{u_1}= g_{u_2}$ $\mu$-a.e. on $\Omega'$ and so, on $\Omega$.
\end{proof} 

\begin{corollary}[Uniqueness of $(p,q)$-minimizers]\label{unicitycor} Let $u_1, u_2 \in N^{1,q}_{loc}({\Omega})$,with $u_2-u_1 \in  N^{1,q}_0({\Omega})$,  be  $(p,q)$-minimizers on $\Omega$. Then $u_1= u_2$ $\mu$-a.e. in $\Omega$.
\end{corollary}
\begin{proof}
Since $u_1$ and $u_2$ are $(p,q)$-minimizers, by Theorem \ref{uniqueness} we have that $g_{u_1}=g_{u_2}$. We are only left to prove that $g_{u_1-u_2}=0$. The Sobolev inequality for Sobolev functions with zero boundary values, Corollary \ref{RemarkPoincare}, would then yield $\Vert u_1-u_2\Vert_{L^{2}(\Omega)}=0$ and hence $u_1=u_2$ $\mu$-a.e in $\Omega$.\\
To show that $g_{u_1-u_2}=0$ $\mu-$a.e in $\Omega$, let $c\in\mathbb{R}$ and $u=\max\lbrace u_1,\min\lbrace u_2,c\rbrace\rbrace$. Then, $u\in N^{1,q}(\Omega)$. We observe that 
$0 \leq u-u_1 \leq \max\{u_1, u_2\} -u_1 = \max\{0, u_2-u_1\} \in N^{1,q}_0(\Omega),$ thus $u-u_1 \in N^{1,q}_0(\Omega)$. Let $V_c=\lbrace x\in\Omega : u_1(x)<c<u_2(x)\rbrace\subseteq\lbrace x\in\Omega: u(x)=c\rbrace$ and hence $g_{u}=0$ $\mu$-a.e. in $V_c$.
Using the $(p,q)$-minimizing definition for $u_1$ and since $g_u=g_{u_1}=g_{u_2}$ $\mu$-a.e. in $\Omega\setminus V_c$, then we get that
\begin{align*}
\int_\Omega a g_{u_1}^pd\mu+\int_\Omega b g_{u_1}^q d\mu&\leq \int_\Omega a g_{u}^pd\mu+\int_\Omega b g_{u}^q d\mu= \int_{\Omega\setminus V_c} ag_{u}^pd\mu+ \int_{\Omega\setminus V_c}b g_{u}^q d\mu\\
&=\int_{\Omega\setminus V_c}
ag_{u_1}^pd\mu+\int_{\Omega\setminus V_c}b g_{u_1}^q d\mu.
\end{align*}
Therefore, $g_{u_1}=g_{u_2}=0$ $\mu$-a.e. in $V_c$ for all $c\in\mathbb{R}$. Notice that in this step we use the assumption $a,b>0$.
Now, 
\begin{equation*}
\lbrace x\in\Omega: u_1(x)<u_2(x)\rbrace\subset \bigcup_{c\in\mathbb{Q}}V_c
\end{equation*}
and hence $g_{u_2}=g_{u_1}=0$ $\mu$-a.e. in $\lbrace x\in\Omega: u_1(x)<u_2(x)\rbrace$. Analogously $\lbrace x\in\Omega: u_2(x)<u_1(x)\rbrace$. It follows that 
$$
g_{u_1-u_2}\leq(g_{u_1}+g_{u_2})\chi_{\lbrace x\in\Omega: u_1(x)\neq u_2(x)\rbrace}=0
$$
$\mu$-a.e in $\Omega$. Then, the Poincar\'{e} inequality for $N^{1,q}_0(\Omega)$ (see Corollary \ref{RemarkPoincare}) yields $\|u_1-u_2\|_{L^q(\Omega)}=0$ and hence $u_1=u_2$ $\mu$-a.e. in $\Omega$.
\end{proof}

\begin{remark} \label{quasiunicity} Under the assumptions of Corollary \ref{unicitycor},  by Proposition 1.59 in \cite{BB} we get that there is a unique $(p,q)$-minimizer up to sets of capacity zero, that is $u_{1}=u_{2}$, $q$-q.e..
 \end{remark}

Now, if we assume that the function $u$  is continuous, we have the following definition (see Definition 7.7 in \cite{BB}).
\begin{definition}[$(p,q)$-harmonic function]\label{(p,q)-harmonic function}
If $u\in N^{1,q}_{loc}({\Omega})$ is a $(p,q)$-minimizer on $\Omega$ and $u$ is a continuous function, then we say that $u$ is a $(p,q)$-harmonic function.
\end{definition}

\begin{definition}[$(p,q)$-potential]\label{(p,q)-potential}
	If $S$ is a compact subset of a domain $\Omega$ in $X$, then by the
	$(p,q)$-potential for $S$ with respect to $\Omega$ we mean the function $u \in N^{1,q}(X)$ which is a $(p,q)$-minimizer in $\Omega \setminus S$, $u=1$ $q$-q.e. on $S$, and $u = 0$ $q$-q.e. on $X \setminus \Omega$.
\end{definition}

\section{Boundedness property }\label{Sec3}
This section is devoted to show that $(p,q)$-quasiminimizers are essentially locally bounded from above. We first give a De Giorgi type inequality which permits to use the De Giorgi method to give an estimate of the essential supremum of $(p,q)$-quasiminimizers involving their averaged norms in $L^q$, thus we deduce their local boundedness. This result has been proven in \cite{KS} for quasiminimizers of $p$-Dirichlet integrals.

From now on, we denote $S_{k, r}= \{x \in B(y,r) \cap {\Omega}: u(x)>k\},$ where $k \in \mathbb{R}$ and $r>0$.  Also, for every $y \in \Omega$, we define $R(y)=\frac{d(y,X\setminus \Omega)}{2}$.
\begin{lemma}[De Giorgi type inequality]\label{lem 4.1}
Let $u\in N^{1,q}_{loc}({\Omega})$ be a $(p,q)$-quasiminimizer.  If $0<\rho<R< R(y)$, then there exists $C$ such that the following De Giorgi type inequality 
\begin{equation}\label{5.4}
\int_{S_{k, \rho}}(ag_{u} ^p+bg_{u} ^q) d \mu	\leq C\left( \frac{1}{(R-\rho)^{p}} \int_{S_{k, R}} a(u-k) ^pd \mu+\frac{1}{(R-\rho)^{q}} \int_{S_{k, R}}b(u-k)^q d \mu \right),
\end{equation}
is satisfied. The constant $C$ depends on $K$, given by Definition \ref{qm}, and $p$.
\end{lemma}

\begin{remark}
We note that inequality \eqref{5.4} is equivalent to the following
\begin{align*}
	\int_{B(y,\rho)}(ag_{u} ^p+bg_{u} ^q) d \mu	& \leq C \Bigg( \frac{1}{(R-\rho)^{p}} \int_{B(y,R)} a(u-k) ^p_+ d \mu \\  & \hspace{1.5cm}+\frac{1}{(R-\rho)^{q}} \int_{B(y,R)}b(u-k)^q_+ d \mu\Bigg),
\end{align*} where $(u-k)_+= \max\{u-k,0\}$.
\end{remark}
\begin{proof}[Proof of Lemma \ref{lem 4.1}]
Let $\tau$ be a $\dfrac{1}{R-\rho}$-Lipschitz cutoff function so that $0\leq \tau \leq1$, $\tau=1$ on $B(y, \rho)$ and the support of $\tau$ is contained in $B(y, R)$.
We consider 
\begin{equation}\label{v}
w=u- \tau (u-k)_+= \begin{cases}
(1-\tau)(u-k)+k\quad \mbox{ in } S_{k, R,}\\
u \quad \quad \quad\quad\quad\quad\quad\quad\mbox{ otherwise.}\\
\end{cases}
\end{equation}
Using Leibniz rule,
\begin{equation}\label{nablav}
g_{w}\leq (u-k)g_{\tau}+(1-\tau)g_u.
\end{equation}
We observe that	
\begin{equation}\label{nablavbis}
g_{w} ^l\leq 
2^{l-1}\left(g_{u} ^l(1-\chi_{S_{k, \rho}})+\left(\dfrac{u-k}{R-\rho}\right)^l\right) \quad \mbox{ in $S_{k, R}$, where $l=p,q$ } .
\end{equation}	
Since $u$ is a $(p,q)$-quasiminimizer, then 
\begin{align}\label{5.7}
\int_{S_{k, \rho}}(ag_{u} ^p+bg_{u} ^q) d \mu  &\leq\int_{S_{k, R}}(ag_{u} ^p+bg_{u} ^q) d \mu \leq K\int_{S_{k, R}}(ag_{w} ^p+bg_{w} ^q) d \mu \nonumber \\
& \leq \frac{K 2^{p-1}}{(R-\rho)^p}\int_{S_{k, R}}a(u-k) ^p \, d\mu+\frac{K 2^{q-1}}{(R-\rho)^q}\int_{S_{k, R}}b(u-k)^q d \mu \nonumber \\
&\quad +K 2^{p}\int_{S_{k, R}\setminus S_{k, \rho}}(g_{u} ^p+g_{u} ^q) d \mu.
\end{align}	
\noindent By adding $K 2^{p}\int_{S_{k, \rho}}(ag_{u} ^p+bg_{u} ^q) d \mu $ to both sides of (\ref{5.7}), we get 
\begin{align}\label{5.7bis}
(1&+K 2^{p}) \int_{S_{k, \rho}}(ag_{u} ^p+bg_{u} ^q) d \mu   \leq  \frac{K 2^{p-1}}{(R-\rho)^p}\int_{S_{k, R}}a(u-k) ^p\, d\mu \nonumber \\&
+\frac{K 2^{q-1}}{(R-\rho)^q}\int_{S_{k, R}}b(u-k)^q d \mu + K 2^{p}\int_{S_{k, R}}(ag_{u} ^p+bg_{u} ^q)d \mu.
\end{align}	
Let $\theta= \dfrac{K 2^{p}}{1+K 2^{p}}<1$, then 
\begin{align*}
\int_{S_{k, \rho}}(ag_{u} ^p+bg_{u} ^q)  d \mu  &\leq \frac{\theta}{2(R-\rho)^p}\int_{S_{k R}}a(u-k) ^p\, d\mu\\ &\quad+\frac{\theta}{2(R-\rho)^q}\int_{S_{k, R}}b(u-k)^q d \mu + \theta\int_{S_{k, R}}(ag_{u} ^p+bg_{u} ^q) d \mu.
\end{align*}

At this point we can use Lemma 6.1 of \cite{G}, with $\alpha=p$, $\beta=q$, $$Z(\rho)=\int_{S_{k, \rho}}(ag_{u} ^p+bg_{u} ^q)   d \mu, \quad A=  \frac{\theta}{2}\int_{S_{k, R}}a(u-k) ^pd \mu, \quad B= \frac{\theta}{2}\int_{S_{k, R}}b(u-k)^q d \mu,$$to get
\begin{align*}
	\int_{S_{k, \rho}}(ag_{u} ^p+bg_{u} ^q)   d \mu &	\leq C\Bigg(\frac{\theta }{2(R-\rho)^p}\int_{S_{k, R}}a(u-k) ^p\, d\mu\\ & \hspace{1.6cm} +\frac{\theta }{2(R-\rho)^q}\int_{S_{k, R}}b(u-k)^q d \mu\Bigg).
\end{align*}
That is
\begin{align*}
	\int_{S_{k, \rho}}(ag_{u} ^p+bg_{u} ^q) d \mu	&	\leq C \left( \frac{1}{(R-\rho)^{p}} \int_{S_{k, R}} a(u-k) ^pd \mu+\frac{1}{(R-\rho)^{q}} \int_{S_{k, R}}b(u-k)^q d \mu \right).
\end{align*}
\end{proof}

\begin{remark}\label{Proposition 3.3KS}
Let $\Omega$ be an open subset of $X$. Notice that if $u$ is a $(p,q)$-quasiminimizer then $-u$ is also a $(p,q)$-quasiminimizer. Therefore, by Lemma \ref{lem 4.1},  we get that $-u$ satisfies \eqref{5.4}.
\end{remark}

Now, suppose that $u$ satisfies \eqref{5.4}. Let $0<\frac{R}{2}<\rho<R\leq R(y)$ such that $B(y, R)\subset \Omega$. Then $g_{(u-k)_+}\leq g_u \chi_{S_{k, R}}$ in $B(y, R)$. We have
\begin{align}\label{ks2}
\int_{B\left(y, \frac{R+\rho}{2}\right)} (ag_{(u-k)_+}^p+bg_{(u-k)_+}^q) d \mu&\leq  C \Bigg (\dfrac{1}{(R-\rho)^{p}} \int_{B(y, R)} a(u-k)_+^p \, d\mu\nonumber\\ & \hspace{1.4cm}+ \dfrac{1}{(R-\rho)^{q}} \int_{B(y, R)} b(u-k)_+^q\, d\mu \Bigg).
\end{align}
Let $w=\tau (u-k)_+$. We have $$g_w \leq g_{(u-k)_+}\tau +(u-k)_+ g_{\tau}\leq g_{(u-k)_+}+ \frac{1}{R-\rho}(u-k)_+.$$
Inequality \eqref{ks2} and the boundedness assumption of $a$ and $b$ in $\Omega$ imply 
\begin{align*}
&\int_{B\left(y, \frac{R+\rho}{2}\right)} (ag_w^p + bg_w^q) d \mu\leq  2^{p-1}\int_{B\left(y, \frac{R+\rho}{2}\right)} (ag_{(u-k)_+}^p+ bg_{(u-k)_+}^q ) d \mu\\
&\quad+\frac{2^{p-1}}{(R-\rho)^p}\int_{B\left(y, \frac{R+\rho}{2}\right)} a(u-k)_+^p d \mu + \frac{2^{q-1}}{(R-\rho)^q} \int_{B\left(y, \frac{R+\rho}{2}\right)} b(u-k)_+^qd \mu\\
&\leq C \left(\frac{2^{q-1}}{(R-\rho)^{p}} \int_{B(y, R)} a (u-k)_+^pd \mu+\frac{ 2^{q-1}}{(R-\rho)^{q}} \int_{B(y, R)}b (u-k)_+^qd \mu \right) \\
&\quad+\frac{2^{q-1}}{(R-\rho)^{p}}\int_{B(y, R)}a (u-k)_+^p d \mu+\frac{2^{q-1}}{(R-\rho)^{q}}\int_{B(y, R)} b(u-k)_+^q	 d \mu\\
&=  C \left(\frac{1}{(R-\rho)^{p}} \int_{B(y, R)} a(u-k)_+^p d \mu+ \frac{1}{(R-\rho)^{q}} \int_{B(y, R)}b (u-k)_+^q\, d\mu \right)\\
&\leq C \left(\frac{1}{(R-\rho)^{p}} \int_{B(y, R)} \beta(u-k)_+^p d \mu+ \frac{1}{(R-\rho)^{q}} \int_{B(y, R)}\beta (u-k)_+^q\, d\mu \right)\\
&= C\left(\frac{1}{(R-\rho)^{p}} \int_{B(y, R)} (u-k)_+^p d \mu+ \frac{1}{(R-\rho)^{q}} \int_{B(y, R)}(u-k)_+^q\, d\mu \right)
\end{align*}
where $C$ depends on $\beta$. Therefore, 
\begin{equation}\label{smallcondition}
\int_{B\left(y, \frac{R+\rho}{2}\right)} (ag_w^p + bg_w^q) d \mu\leq 2C\max_{l=p,q}\Big\lbrace\frac{1}{(R-\rho)^{l}}\int_{B(y, R)} (u-k)_+^l d \mu\Big\rbrace
\end{equation}
\noindent By the definition of $w$, the $(1,p)$-Poincar\'e inequality, Lemma \ref{Lemma 2.1 KS} and H\"{o}lder inequality, there is $q<t<p^*$  such that

\begin{align}\label{ks40}
&\left(\dashint_{B(y, \rho)}(u-k)_+^t d\mu\right)^{\frac{l}{t}}= \left(\dashint_{B(y, \rho)}w^t d\mu\right)^{\frac{l}{t}}\nonumber \\
&\quad\leq \left(\dfrac{\mu\big(B\big(y, \frac{R+\rho}{2}\big)\big)}{\mu(B(y, \rho))}\right)^{\frac{l}{t}}\left(\dashint_{B\left(y, \frac{R+\rho}{2}\right)}w^t d\mu\right)^{\frac{l}{t}} \nonumber\\
&\quad\leq\left(\dfrac{\mu\big(B\big(y, \frac{R+\rho}{2}\big)\big)}{\mu(B(y, \rho))}\right)^{\frac{l}{t}} (CR)^l\left(\dashint_{B\left(y, \frac{R+\rho}{2}\right)}g_w^s d\mu\right)^{\frac{l}{s}}\nonumber\\
&\quad\leq\left(\dfrac{\mu\big(B\big(y, \frac{R+\rho}{2}\big)\big)}{\mu(B(y, \rho))}\right)^{\frac{l}{t}} (CR)^l\dashint_{B\left(y, \frac{R+\rho}{2}\right)}g_w^l d\mu\nonumber\\
&\quad\leq\left(\dfrac{\mu\big(B\big(y, \frac{R+\rho}{2}\big)\big)}{\mu(B(y, \rho))}\right)^{\frac{l}{t}} (CR)^l\dashint_{B\left(y, \frac{R+\rho}{2}\right)}(g_w^p+g_w^q) d\mu\nonumber\\
&\quad\leq\left(\dfrac{\mu\big(B\big(y, \frac{R+\rho}{2}\big)\big)}{\mu(B(y, \rho))}\right)^{\frac{l}{t}} CR^l\dashint_{B\left(y, \frac{R+\rho}{2}\right)}(ag_w^p+bg_w^q) d\mu,
\end{align}
where $l=p,q$. Notice that $C$ depends on $C_{PI}$ and $\alpha$. 

Using the doubling property, we can estimate the first term of the right-hand side of \eqref{ks40} in this way $$\left(\dfrac{\mu\big(B\big(y, \frac{R+\rho}{2}\big)\big)}{\mu(B(y, \rho))}\right)^{\frac{l}{t}}\leq C,$$
where $l=p,q$ and $C$ depends on the exponent $Q$ in \eqref{s}.

Using the latter estimate and \eqref{ks40}, we obtain
\begin{equation*}
	\left(\dashint_{B(y, \rho)}(u-k)_+^t d\mu\right)^{\frac{l}{t}}
	\leq CR^l\dashint_{B\left(y, \frac{R+\rho}{2}\right)}(ag_w^p+bg_w^q) d\mu,
\end{equation*}
for both $l=p,q$. Therefore,
\begin{equation}\label{ks40bis}
	\max_{l=p,q}\Big\lbrace\frac{1}{R^{l}}\left(\dashint_{B(y, \rho)}(u-k)_+^t d\mu\right)^{\frac{l}{t}}\Big\rbrace
	\leq C\dashint_{B\left(y, \frac{R+\rho}{2}\right)}(ag_w^p+bg_w^q) d\mu.
\end{equation}

Now, using \eqref{ks40bis} and then \eqref{smallcondition}, we get
\begin{align}\label{ks4}
\max_{l=p,q}\Big\lbrace\frac{1}{R^{l}}&\left(\dashint_{B(y, \rho)}(u-k)_+^t d\mu\right)^{\frac{l}{t}}\Big\rbrace\leq C\dashint_{B\left(y, \frac{R+\rho}{2}\right)}(ag_w^p+bg_w^q) d\mu \nonumber \\
&\leq \frac{C}{\mu(B(y, \frac{R+\rho}{2}))}   \max_{l=p,q}\Big\lbrace\frac{1}{(R-\rho)^{l}}\int_{B(y, R)} (u-k)_+^l d \mu\Big\rbrace \nonumber \\
&\leq  C\max_{l=p,q}\Big\lbrace\frac{1}{(R-\rho)^{l}} \dashint_{B(y, R)} (u-k)_+^l d \mu\Big\rbrace.
\end{align}
For both $l=p,q$, by H\"{o}lder inequality, we have
\begin{align*}
\frac{1}{R^{l}}\dashint_{B(y, \rho)}(u-k)_+^l d\mu &\leq\frac{1}{R^{l}} \left (\dashint_{B(y, \rho)}(u-k)_+^t d\mu\right)^{\frac{l}{t}}\left(\dfrac{\mu(S_{k, \rho})}{\mu(B(y, \rho)}\right)^{1-\frac{l}{t}}\\
&\leq \frac{1}{R^{l}}\left (\dashint_{B(y, \rho)}(u-k)_+^t d\mu\right)^{\frac{l}{t}}\left(\dfrac{\mu(S_{k, \rho})}{\mu(B(y, \rho)}\right)^{1-\frac{p}{t}}\\
&\leq\max_{l=p,q}\Big\lbrace\frac{1}{R^{l}}\left (\dashint_{B(y, \rho)} (u-k)_+^td \mu\right)^{\frac{l}{t}}\Big\rbrace \left(\dfrac{\mu(S_{k, \rho})}{\mu(B(y, \rho)}\right)^{1-\frac{p}{t}}.
\end{align*}
Now, using \eqref{ks4}, we obtain
\begin{equation}\label{ks6}
\frac{1}{R^{l}}\dashint_{B(y, \rho)}(u-k)_+^l d\mu \leq C\max_{l=p,q}\Big\lbrace\frac{1}{(R-\rho)^{l}} \dashint_{B(y, R)} (u-k)_+^l d \mu\Big\rbrace  \left(\dfrac{\mu(S_{k, \rho})}{\mu(B(y, \rho)}\right)^{1-\frac{p}{t}},
\end{equation}
and this holds for both $l=p$ and $l=q$. \\
Let $h<k$. Then, for $l=p,q$,
\begin{align}\label{ks7}
(k-h)^{l}\mu(S_{k, \rho})&= \int_{S_{k, \rho}}(k-h)^l d\mu\leq \int_{S_{k, \rho}}(u-h)^l d\mu\nonumber\\
&\leq\int_{S_{h, \rho}}(u-h)^l d\mu=\int_{B(y,\rho)}(u-h)_{+}^{l}d\mu.
\end{align}
For $l=p,q$, we define
$$u_{l}(h,\rho)=\left(\dashint_{B(y, \rho)}(u-h)_+^l d\mu\right)^{\frac{1}{l}}.$$
By \eqref{ks7} we get
\begin{equation*}
\mu(S_{k, \rho})\leq \dfrac{\mu(B(y, \rho))}{(k-h)^{l}}u_{l}(h,\rho)^{l}\leq \dfrac{\mu(B(y, R))}{(k-h)^{l}}u_{l}(h,R)^{l}.
\end{equation*}
with $l=p,q$. As a consequence, we observe that
\begin{equation}\label{ipq}
\mu(S_{k, \rho})^{\left(\frac{1}{l}-\frac{p}{lt}\right)}
\leq  \mu(B(y, R))^{\left(\frac{1}{l}-\frac{p}{lt}\right)} (k-h)^{-\theta}u_l(h,R)^{\theta},
\end{equation} where $\theta=1-\frac{p}{t}>0$ and both $l=p,q$.\\
\noindent
Now, if $\displaystyle\max_{l=p,q}\Big\lbrace\frac{1}{(R-\rho)^{l}} \dashint_{B(y, R)} (u-k)_+^l d \mu\Big\rbrace =\frac{1}{(R-\rho)^{p}} \dashint_{B(y, R)} (u-k)_+^p d \mu$, then by \eqref{ks6}, the doubling property and \eqref{ipq}, the following inequalities hold
\begin{align*}
u_p(k, \rho)&\leq \dfrac{CR}{R- \rho}u_{p}(k,R)\left(\dfrac{\mu(S_{k, \rho})}{\mu(B(y, \rho))}\right)^{\frac{1}{p}-\frac{1}{t}} \nonumber\\
&\leq \dfrac{CR}{R- \rho}u_{p}(k,R)\left(\dfrac{\mu(B(y,R))}{\mu(B(y, \rho))}\right)^{\frac{1}{p}-\frac{1}{t}}(k-h)^{-\theta}u_{p}(h, R)^{\theta}\nonumber\\
&\leq \dfrac{CR}{R- \rho}(k-h)^{-\theta}u_p(h,R)^{1+\theta}.
\end{align*}
Analogously, if $\displaystyle\max_{l=p,q}\Big\lbrace\frac{1}{(R-\rho)^{l}} \dashint_{B(y, R)} (u-k)_+^l d \mu\Big\rbrace =\frac{1}{(R-\rho)^{q}} \dashint_{B(y, R)} (u-k)_+^q d \mu$, proceeding in the same way as above, we get
\begin{equation*}
u_q(k, \rho)\leq \dfrac{CR}{R- \rho}(k-h)^{-\theta}u_q(h,R)^{1+\theta},
\end{equation*}

Therefore, we obtain
\begin{equation}\label{ks8}
u_l(k, \rho)\leq \dfrac{CR}{R- \rho}(k-h)^{-\theta}u_l(h,R)^{1+\theta},
\end{equation}
for either $l=p$ or $l=q$, where $C$ depends on $\alpha,\beta, C_{PI}$, $Q$ and $\theta=1-\frac{p}{t}$. 
At this point, using inequality \eqref{ks8}, we obtain the next lemma.

\begin{lemma}\label{Proposition 4.1 KS}
For all $k_0\in \mathbb{R}$ we have that $$u_p\left(k_0+d, \frac{R}{2}\right)=0,$$ where \begin{equation}\label{ksd}
d=(C\, 2^{2+\theta+\frac{1+\theta}{\theta}})^{\frac{1}{\theta}}u_q(k_0, R),
\end{equation} with $\theta$ and $C$ given as in \eqref{ks8}.
\end{lemma}
\begin{proof}
For $n\in \mathbb{N}\cup \{0\}$, let $\rho_{n}=\frac{R}{2}\left(1+\frac{1}{2^{n}}\right)\leq R$ and $k_{n}=k_{0}+d\left(1-\frac{1}{2^n}\right)$, where $d>0$ will be chosen later. Then $\rho_{0}=R$, $\rho_{n}\searrow \frac{R}{2}$ and $k_n\nearrow k_0+d$.

We now show by induction that with a suitable $d$
\begin{equation}\label{induction}
u_l(k_n,\rho_n)\leq 2^{-\tau n}u_q(k_0, R)
\end{equation}
with $\tau=\frac{1+\theta}{\theta}$ and for either $l=p$ or $l=q$ (this depends on whether \eqref{ks8} holds for $l=p$ or $l=q$). Therefore, independently of the case, we would have
$$
0\leq u_p\left(k_0+d,\frac{R}{2}\right)\leq u_l\left(k_0+d, \frac{R}{2}\right)\leq u_l(k_n, \rho_n)\leq 2^{-\tau n}u_q(k_0, R)\rightarrow 0
$$
as $n\rightarrow +\infty.$\\
Indeed, \eqref{induction} is trivially true for $n=0$. Now assuming \eqref{induction} for $n\geq 0$, by \eqref{ks8} with $R=\rho_n$, $\rho=\rho_{n+1}$, $k=k_{n+1}$ and $h=k_n$, we get
\begin{align*}
u_l(k_{n+1}, \rho_{n+1})&\leq C\left(\frac{\rho_n}{\rho_n-\rho_{n+1}}\right)(k_{n+1}-k_n)^{-\theta}u_l(k_n, \rho_n)^{1+\theta}\\
&\leq C\left(\frac{R}{\rho_n-\rho_{n+1}}\right)(k_{n+1}-k_n)^{-\theta}\left(2^{-\tau n}u_q(k_0, R)\right)^{1+\theta}\\
&\leq\frac{2^{n+2} C}{(2^{-n-1}d)^{\theta}}\left(2^{-\tau n}u_q(k_0, R)\right)^{1+\theta}=2^{-\tau (n+1)}u_q(k_0, R)
\end{align*}
provided that $d^{\theta}=C 2^{2+\theta+\tau}u_q(k_0, R)^{\theta}$. Therefore, $u_p\left(k_0+d, \frac{R}{2}\right)=0$ as wanted. 
\end{proof}

We give the following weak Harnack inequality Theorem, which states that a function satisfying the De Giorgi type inequality \eqref{5.4} is essentially locally bounded from above. 
\begin{theorem}[Weak Harnack inequality]\label{Theorem 4.2}
We consider an open subset $\Omega \subset X$ such that $B(y, R) \subset \Omega$, with $0 < R \leq R(y)$, and $k_0 \in \mathbb{R}$. If $u$ satisfies \eqref{5.4}, then there exists $C>0$ such
that $$\esssup_{B\left(y, \frac{R}{2}\right)} u \leq k_0 + C \left(\dashint_{B(y, R)}(u-k_0)_+^q \, d\mu\right)^{\frac{1}{q}}.$$
The constant $C=C(\alpha, \beta, C_{PI}, Q)$ is independent of the ball $B(y, R)$.
\end{theorem}
\begin{proof}
The proof can be easily deduced from the previous lemma. In fact, Lemma \ref{Proposition 4.1 KS} implies that $$u_p\left(k_0 + d, \frac{R}{2}\right) = 0,$$with $d$ given by \eqref{ksd}. Consequently, $u\leq k_0+d$ almost everywhere in $B(y, \frac{R}{2})$, meaning
$$\esssup_{B(y, \frac{R}{2})}u \leq k_0 +d= k_0+ C \left(\dashint_{B(y, R)}(u-k_0)_+^q \, d\mu\right)^{\frac{1}{q}}.$$

\end{proof}
\section{H\"{o}lder continuity}\label{Sec4}
In this section we prove a locally H\"{o}lder continuity result for $(p,q)$-quasiminimizers.
We first note that if $u$ is a $(p,q)$-quasiminimizer in $\Omega$, then by Lebesgue's differentiation theorem, there exists a function
\begin{equation}\label{u tilde}
	\tilde{u}(x)= \limsup_{\rho \to 0} \dashint_{B(x,\rho)}u \, d\mu
\end{equation}
which is equal to $u$ $\mu$-a.e on $\Omega$.\\
\noindent If $B(y, \rho) \subset \Omega$, we define
$m(\rho) = \essinf_{B(y,\rho)}u$ and $M(\rho) = \esssup_{B(y,\rho)}u.$
 From the definition of $\tilde{u}$ (see \cite{BB}, p. 202), we have that $$m(\rho)= \inf_{B(y,\rho)} \tilde{u}  \quad \mbox {and}\quad M(\rho) = \sup_{B(y,\rho)} \tilde{u}.$$
The main theorem of this section implies that the class of $(p,q)$-harmonic functions is nontrivial and that every $(p,q)$-harmonic function is locally H\"{o}lder continuous in the domain of its $(p,q)$-harmonicity.
\begin{theorem}[Local H\"{o}lder continuity]\label{Theorem 5.2}
	Let $u$ be a $(p,q)$-quasiminimizer, $\tilde{u}$ defined by \eqref{u tilde} and $0 < \rho < R <\frac{R(y)}{\lambda'}$ with $B(y, 2\lambda' R) \subset \Omega$. Then there exists $0 < \eta < 1$ such that
	$$\osc(\tilde{u},B(y, \lambda' \rho)) \leq 4^\eta
	\left(\frac{\rho }{R}\right)^\eta \osc(\tilde{u},B(y, \lambda' R)),$$
	where $\osc(\tilde{u},B(y, \cdot)) = \sup_{B(y,\cdot)}\tilde{u}- \inf_{B(y,\cdot)}\tilde{u}$ is the oscillation of $\tilde{u}$. In particular, $\tilde{u}$ is locally H\"{o}lder continuous on $\Omega$, that is $u$ can be modified on a set of capacity zero so that it becomes locally H\"{o}lder continuous on $\Omega$.
\end{theorem}

Before proving Theorem \ref{Theorem 5.2}, we need to premise some auxiliary results. 

Now, suppose that $u$ satisfies the De Giorgi type inequality \eqref{5.4}. Let $0 < \rho < R <\frac{R(y)}{\lambda'}$ be such
that $B(y, 2\lambda'R)\subset \Omega$, where $\lambda'$ is as in \eqref{(2.8)}. 
We assume that $h>0$ is such that
\begin{equation} \label{4.0ineq}
	\mu(S_{h, R}) \leq \gamma \mu(B(y, R)) \quad \mbox{for some $\gamma\in (0,1)$.}
\end{equation}

For $k> h$, we define
\begin{equation}\label{v(x)}
	v = \min\{u, k\} - \min\{u, h\}.
\end{equation} note that, by \eqref{v(x)}, we have
\begin{equation} \label{4.0bisineq}
	v=\begin{cases}
		0 \hspace{1.98cm} \mbox{if $u\leq h<k$,}\\
		u-h \hspace{1.3cm} \mbox{if $h<u<k$,}\\
		k-h \hspace{1.3cm} \mbox{if $u\geq k>h$.}
	\end{cases}
\end{equation}
We remark that $v \in N^{1,q}_{loc}(\Omega)$, because $u \in N^{1,q}_{loc}(\Omega)$.
From \eqref{4.0ineq} and \eqref{4.0bisineq}, we have that $\mu(\{x \in B(y, R): v(x) > 0\}) \leq \gamma\mu(B(y, R)).$ Using H\"{o}lder inequality and Lemma \ref{Lemma 2.1 KS} with $t = s<p<q$, we obtain
\begin{align*}
	(k&-h) \mu(S_{k,R})= \int_{S_{k,R}} v\,d\mu \leq \int_{B(y,R)}v \,d\mu \nonumber\\
	&  \leq \mu(B(y,R))^{1-\frac{1}{s}}  \left(\int_{B(y, R)}v^s \,d\mu \right)^{\frac{1}{s}} \nonumber\\
	& \leq C R \, \dfrac{\mu(B(y,R))}{(\mu(B(y,\lambda' R)))^{\frac{1}{s}}} \left(\int_{B(y,\lambda'  R)}g_v^s \,d\mu \right)^{\frac{1}{s}} \nonumber\\
	& \leq CR \, \mu(B(y,R))^{1-\frac{1}{s}} \left(\int_{B(y, \lambda' R)}g_v^s \,d\mu \right)^{\frac{1}{s}}\nonumber\\
	&=CR \, \mu(B(y,R))^{1-\frac{1}{s}} \left(\int_{S_{h, \lambda' R}\setminus S_{k, \lambda' R}}g_u^s \,d\mu \right)^{\frac{1}{s}}\nonumber\\
	& \leq  C R \,\mu(B(y,R))^{1-\frac{1}{s}}\Big(\mu(S_{h, \lambda' R})-\mu(S_{k, \lambda' R})\Big)^{\frac{1}{s}-\frac{1}{l}} \left(\int_{S_{h, \lambda' R}}g_u^l \,d\mu \right)^{\frac{1}{l}}\nonumber\\
	& \leq CR \, \mu(B(y,R))^{1-\frac{1}{s}}\Big(\mu(S_{h, \lambda' R})-\mu(S_{k, \lambda' R})\Big)^{\frac{1}{s}-\frac{1}{l}} \left(\int_{S_{h, \lambda' R}}(g_u^p+g_u^q) \,d\mu \right)^{\frac{1}{l}}\nonumber\\
	& \leq C R \, \mu(B(y,R))^{1-\frac{1}{s}}\Big(\mu(S_{h, \lambda' R})-\mu(S_{k, \lambda' R})\Big)^{\frac{1}{s}-\frac{1}{l}} \left(\int_{S_{h, \lambda' R}}(ag_u^p+bg_u^q) \,d\mu \right)^{\frac{1}{l}}\nonumber\\
	& = CR \, \mu(B(y,R))^{1-\frac{1}{s}}\Big(\mu(S_{h, \lambda' R})-\mu(S_{k, \lambda' R})\Big)^{\frac{1}{s}-\frac{1}{l}} \left(\int_{S_{h, \lambda' R}}(ag_u^p+bg_u^q) \,d\mu \right)^{\frac{1}{l}},
\end{align*}
where $l=p, q$ and $C$ depends on $\alpha$, $\gamma$ and $C_d$. By De Giorgi type inequality \eqref{5.4}, we get
\begin{align}\label{beforemax}
	(k-h)^l \mu(S_{k,R})^l &\leq CR^l\mu(B(y,R))^{l-\frac{l}{s}}\Big(\mu(S_{h, \lambda' R})-\mu(S_{k, \lambda' R})\Big)^{\frac{l}{s}-1}\nonumber\\ 
	& \quad \cdot\left(\frac{1}{(\lambda ' R)^p}\int_{S_{h, 2\lambda' R}}a(u-h)^p\,d\mu + \frac{1}{(\lambda' R)^q}\int_{S_{h, 2\lambda ' R}}b(u-h)^q \,d\mu \right)\nonumber\\
	&\leq CR^l\mu(B(y,R))^{l-\frac{l}{s}}\Big(\mu(S_{h, \lambda' R})-\mu(S_{k, \lambda' R})\Big)^{\frac{l}{s}-1}\nonumber\\ & \quad \cdot\left(\frac{1}{R^p}\int_{S_{h, 2\lambda' R}}a(u-h)^p\,d\mu + \frac{1}{R^q}\int_{S_{h, 2\lambda ' R}}b(u-h)^q \,d\mu \right)\nonumber\\
	&\leq CR^l\mu(B(y,R))^{l-\frac{l}{s}}\Big(\mu(S_{h, \lambda' R})-\mu(S_{k, \lambda' R})\Big)^{\frac{l}{s}-1}\nonumber\\ & \quad \cdot\left(\frac{1}{R^p}\int_{S_{h, 2\lambda' R}}(u-h)^p\,d\mu + \frac{1}{R^q}\int_{S_{h, 2\lambda ' R}}(u-h)^q \,d\mu \right),
\end{align} where $C$ also depends on $\beta$. For $l=p,q$ we define
$$
M_{l}=\left(\frac{(k-h)\mu(S_{k, R})}{CR\mu(B(y,R))^{1-\frac{1}{s}}\Big(\mu(S_{h, \lambda' R})-\mu(S_{k, \lambda' R})\Big)^{\frac{1}{s}-\frac{1}{l}}}\right)^{l}.
$$
By inequality \eqref{beforemax} we deduce that
\begin{align*}
M_l\leq\frac{1}{R^p}\int_{S_{h, 2\lambda' R}}(u-h)^p\,d\mu + \frac{1}{R^q}\int_{S_{h, 2\lambda ' R}}(u-h)^q \,d\mu,
\end{align*}for $l=p,q$.

Therefore,
if $\displaystyle\max_{l=p,q}\Big\lbrace\frac{1}{R^l}\int_{S_{h, 2\lambda' R}}(u-h)^l\,d\mu \Big\rbrace=\frac{1}{R^p}\int_{S_{h, 2\lambda' R}}(u-h)^p\,d\mu$, we have
\begin{align*}
	(k-h)\mu(S_{k,R})\leq C\mu(B(y,R))^{1-\frac{1}{s}}&\Big(\mu(S_{h, \lambda' R})-\mu(S_{k, \lambda' R})\Big)^{\frac{1}{s}-\frac{1}{p}}\\
	&\cdot\left( \int_{S_{h, 2\lambda' R}}(u-h)^p\,d\mu\right)^{\frac{1}{p}}.
\end{align*}

Analogously, if $\displaystyle\max_{l=p,q}\Big\lbrace\frac{1}{R^l}\int_{S_{h, 2\lambda' R}}(u-h)^l\,d\mu \Big\rbrace=\frac{1}{R^q}\int_{S_{h, 2\lambda' R}}(u-h)^q\,d\mu$, we get
\begin{align*}
	(k-h)\mu(S_{k,R})\leq C\mu(B(y,R))^{1-\frac{1}{s}}&\Big(\mu(S_{h, \lambda' R})-\mu(S_{k, \lambda' R})\Big)^{\frac{1}{s}-\frac{1}{q}}\\
	&\cdot\left( \int_{S_{h, 2\lambda' R}}(u-h)^q\,d\mu\right)^{\frac{1}{q}}.
\end{align*}
Thus, depending on the case we are, there exists $C=C(\alpha, \beta, \gamma, C_d)>0$ such that 
\begin{align}\label{twocases}
	(k-h)\mu(S_{k,R})\leq C\mu(B(y,R))^{1-\frac{1}{s}}&\Big(\mu(S_{h, \lambda' R})-\mu(S_{k, \lambda' R})\Big)^{\frac{1}{s}-\frac{1}{q}}\nonumber\\
	&\cdot\left( \int_{S_{h, 2\lambda' R}}(u-h)^q\,d\mu\right)^{\frac{1}{q}},
\end{align}
is satisfied for either $l=p$ or $l=q$. 

At this point, using inequality \eqref{twocases}, we are ready to give the following lemma that is needed for the proof of Theorem \ref{Theorem 5.2}.

\begin{lemma}\label{Proposition 5.1}
Let $M =M(2\lambda'R)$, $m = m(2\lambda'R)$ and $k_0 = \frac{M+m}{2}$.
If $u$ satisfies \eqref{5.4},  it is locally bounded from below and $\mu(S_{k_0,R}) \leq \gamma \mu(B(y, R))$
for some $0 < \gamma < 1$, then $$\lim_{k \to M}\mu(S_{k,R})=0.$$
\end{lemma}
\begin{proof}
Let $k_j=M-2^{-(j+1)} (M-m)$, $j\in \mathbb{N} \cup \{0\}$. Therefore, $\lim_{j \to +\infty}k_j = M$ and $k_0=\frac{M+m}{2}.$ 
Then, we get $M-k_{j-1}=2^{-j} (M-m)$ and $k_j-k_{j-1}=2^{-(j+1)} (M-m)$. We remark that $u-k_{j-1}\leq M-k_{j-1}$ on $S_{k_{j-1}, \lambda' R}$.  By \eqref{twocases}, for  either $l=p$ or $l=q$, we deduce that
\begin{align*}
 (k_j-k_{j-1}) \mu(S_{k_j,R})  &\leq C  \mu(B(y,R))^{1-\frac{1}{s}} \Big(\mu(S_{k_{j-1}, \lambda' R})-\mu(S_{k_j, \lambda' R})\Big)^{\frac{1}{s}-\frac{1}{l}} \\&\hspace{2.6cm} \cdot \left(\int_{S_{k_{j-1}, 2\lambda' R}}(M-k_{j-1})^i \,d\mu \right)^{\frac{1}{l}}.
\end{align*}
Thus,
\begin{align*}
2^{-(j+1)} (M-m) \mu(S_{k_j,R}) \leq C \mu(B(y,R))^{1-\frac{1}{s}+\frac{1}{l}} \Big(\mu(S_{k_{j-1}, \lambda' R})&-\mu(S_{k_j, \lambda' R})\Big)^{\frac{1}{s}-\frac{1}{l}}\\&  \cdot2^{-j} (M-m).
\end{align*}
If $n>j $, then $\mu(S_{k_n,  R})\leq \mu(S_{k_j,  R})$, and so
\begin{align*}
\mu(S_{k_n,  R})\leq  C \mu(B(y,R))^{1-\frac{1}{s}+\frac{1}{l}} \Big(\mu(S_{k_{j-1}, \lambda' R})-\mu(S_{k_j, \lambda' R})\Big)^{\frac{1}{s}-\frac{1}{l}}.
\end{align*}
By summing the above inequality over $j \in (j_0, n)$, we get
\begin{align}\label{5.3KS}
n\mu(S_{k_n,  R})^{\frac{sl}{l-s}}&\leq C \mu(B(y,R))^{\frac{sl-l+s}{(l-s)}}\,\Big(\mu(S_{k_{0}, \lambda' R})-\mu(S_{k_n, \lambda' R})\Big)\nonumber \\
&\leq C\mu(B(y,R))^{\frac{sl}{(l-s)}}.
\end{align}
Hence, independently of the case,  $\lim_{n\to+\infty}\mu(S_{k_n  R})=0$ and, since 
$\mu(S_{k,  R})$ is a monotonic decreasing function of $k$, we conclude that $\lim_{k \to M}\mu(S_{k,R})=0$.
\end{proof}

\begin{remark}\label{indipendence of n}
From inequality \eqref{5.3KS}, we have that $$\frac{\mu(S_{k_n,  R})}{\mu(B(y,R))}\leq C n^{-\xi} \to 0, \quad \mbox{as $n \to +\infty$},$$ uniformly with respect to $u$, where $\xi>0$.
\end{remark}
\begin{remark}\label{-uinsteadofu}
We note that it is not restrictive to suppose that 
\begin{equation}\label{muSalpha}
\mu(S_{k, R})\leq \frac{\mu(B(y, R))}{2} \quad\mbox{for all $k \in \mathbb{R}$, $R>0$ and $y \in \Omega$ with $B(y, R) \subset \Omega$.}
\end{equation}
 In fact, if $$\mu(S_{k, R})=\mu(\{x \in B(y, R):u(x)>k\})> \frac{\mu(B(y, R))}{2},$$ then
$$\mu(\{x \in B(y, R):-u(x)\leq - k\})>\frac{\mu(B(y, R))}{2},$$
and so, $$\mu(\{x \in B(y, R):-u(x)>- k\})<\frac{\mu(B(y, R))}{2}.$$
That is, inequality \eqref{muSalpha} holds true considering $-u$ instead of $u$.
 \end{remark}

\begin{proof}[Proof of Theorem \ref{Theorem 5.2}]
Firstly, we observe that ${u}$ and $-{u}$ satisfy \eqref{5.4} (see Remark \ref{Proposition 3.3KS}).
We consider $k_0 =\frac{M + m}{2}$ where $M$ and $m$ are as in Lemma \ref{Proposition 5.1}. As we have pointed out in Remark \ref{-uinsteadofu}, it is not restrictive to suppose that 
\begin{equation}\label{hypothesisRemark}
\mu(S_{k_0, R})\leq \frac{\mu(B(y, R))}{2}.
\end{equation} 
Using Theorem \ref{Theorem 4.2} replacing $k_0$ with $k_n = M - 2^{-n-1}(M - m)$, $n \in \mathbb{N}\cup \{0\}$, we have
\begin{equation}\label{inequM}
M\left(\frac{\lambda' R}{2}\right) \leq k_n + C (M(2\lambda' R)-k_n) \left(\frac{\mu(S_{k_n,R})}{\mu(B(y,R))}\right)^{\frac{1}{q}},
\end{equation} with $C$ as in Theorem \ref{Theorem 4.2}.

Inequality \eqref{hypothesisRemark} ensures that we can use Remark \ref{indipendence of n}, so it is possible to choose an integer
$n$, independent from $B(y, R)$ and $u$, large enough such that
$$C\left(\frac{\mu(S_{k_n,R})}{\mu(B(y,R))}\right)^{\frac{1}{q}}< \frac{1}{2}.$$
Thanks to this choice, from \eqref{inequM} we deduce the following inequality.
$$M\left(\frac{\lambda' R}{2}\right)<M(2\lambda' R)-(M(2\lambda' R)-m(2\lambda' R))2^{-(n+2)},$$
thus,
\begin{align*}
M\left(\frac{\lambda' R}{2}\right)-m\left(\frac{\lambda' R}{2}\right)&\leq M\left(\frac{\lambda' R}{2}\right)-m(2\lambda' R)\\
&\leq (M(2\lambda' R)-m(2\lambda' R))(1-2^{-(n+2)}).
\end{align*}
As a consequence of the previous inequality, we can write
\begin{equation}\label{5.5KS}
\osc \left(\tilde{u}, B\left(y, \frac{\lambda' R}{2}\right)\right)< \tau \osc(\tilde{u}, B(y, 2\lambda' R)),
\end{equation}
where $\tau=1-2^{-(n+2)}<1$.
Now, we consider an index $j \geq 1$ such that $$4^{j-1} \leq\frac{R}{\rho}  < 4^j.$$ Then, from inequality \eqref{5.5KS}, we get
$$\osc(\tilde{u}, B(y, \lambda' \rho)) \leq \tau^{j-1} \osc(\tilde{u}, B(y, \lambda'  4^{j-1}\rho)) \leq \tau^{j-1} \osc(\tilde{u}, B(y, \lambda' R)).$$
We observe that $\tau= 4^{\log_4 \tau}= 4^{\frac{\log \tau}{\log 4}}$ and we deduce that
$$ \tau^{j-1}=4^{\frac{\log \tau}{log 4}(j-1)}=4^{-(j-1)\left(-\frac{\log \tau}{log 4}\right)}=\left(\frac{4}{4^j}\right)^{-\frac{\log \tau}{\log 4}}\leq\left(\frac{4}{\frac{R}{\rho}}\right)^{-\frac{\log \tau}{\log 4}}= 4^{\eta} \left(\frac{R}{\rho}\right)^{-\eta} ,$$ where $\eta=-\frac{\log\tau}{\log 4}< 1$. At the end, we obtain
$$\osc(\tilde{u}, B(y, \lambda' \rho)) \leq 4^{\eta} \left(\frac{R}{\rho}\right)^{-\eta} \osc(\tilde{u}, B(y, \lambda' R)) ,$$ that completes the proof. 
\end{proof}

\section{Harnack inequality}\label{Sec5}
The aim of this section is to prove Harnack inequality for $(p,q)$-quasiminimizers.
\begin{theorem}[Harnack inequality]\label{Corollary 7.3 KS}
	Suppose that $u > 0$, $u$ and $-u$ satisfy \eqref{5.4}. Then there
	exists a constant $C \geq 1$ so that
	$$\esssup_{B(y,R)} u \leq C\,\essinf_{B(y,R)} u$$
	for every ball $B(y,R)$ for which $B(y, 6R) \subset \Omega$ and $0 < R \leq \frac{R(y)}{6}$. Here the constant $C$ is independent of the function $u$.
\end{theorem}
In order to prove Theorem \ref{Corollary 7.3 KS}, we need to premise the following two lemmata.\\
Let us define $D_{\tau, R} = \{x \in B(y, R): u(x) < \tau\}$.
\begin{lemma}\label{Lemma6.1KS}
Let $0<R\leq R(y)$ be such that $B(y, R)\subset \Omega$. Let $\tau>0$, $u\geq 0$ and $-u$ satisfy \eqref{5.4}.
Then there exists $\gamma_0 \in (0,1)$, independent of the ball $B(y, R)$, such that if \begin{equation}\label{gamma0}\mu(D_{\tau, R})\leq \gamma_0 \mu(B(y,R)),
\end{equation} then $$\essinf_{B\left(y,\frac{R}{2}\right)} u \geq \frac{\tau}{2}.$$
\end{lemma}
\begin{proof}
We apply Theorem \ref{Theorem 4.2} to $-u$ and $k_0=-\tau$ and we get
$$\esssup_{B\left(y,\frac{R}{2}\right)} (-u) \leq -\tau +C \left(\frac{1}{\mu(B(y, R))} \int_{D_{\tau, R}}(-u+\tau)^q\, d\mu\right)^{\frac{1}{q}} .$$
Since $\tau-u\leq \tau$, we have that 
\begin{align*}
\essinf_{B\left(y,\frac{R}{2}\right)} u &\geq \tau - C \left(\frac{1}{\mu(B(y, R))}\int_{D_{\tau, R}}(\tau-u)^q\, d\mu\right)^{\frac{1}{q}}	\\
&\geq \tau - C \tau \left(\frac{\mu(D_{\tau, R})}{\mu(B(y,R))}\right)^{\frac{1}{q}}.
\end{align*}
We can choose $\gamma_0= (2C)^{-q}$  in \eqref{gamma0},
 so the proof is completed.
\end{proof}

\begin{lemma}\label{Lemma6.2KS}
Suppose that the hypotheses of Lemma \ref{Lemma6.1KS} hold. For every $\gamma$ with $0 < \gamma < 1$ there is a constant $\lambda > 0$ such that if $\mu(D_{\tau, R})\leq \gamma \mu(B(y,R))$, then $$\essinf_{B\left(y,\frac{R}{2}\right)} u \geq \lambda \tau.$$
\end{lemma}
\begin{proof}
We consider $k,h > 0$ such that $-k > -h$. Using inequality \eqref{twocases} with $-u$, $-k$ and $-h$ instead of $u$, $k$ and $h$, respectively, we obtain the following inequality.
\begin{align*}
(h-k) \mu(D_{k,R})& \leq  C \mu(B(y,R))^{1-\frac{1}{s}} \Big(\mu(D_{h, \lambda' R})-\mu(D_{k, \lambda' R})\Big)^{\frac{1}{s}-\frac{1}{l}} \\ &\hspace{3.25cm} \cdot \left(\int_{D_{h, 2\lambda' R}}(h-u)^l \,d\mu \right)^{\frac{1}{l}},
\end{align*}
for either $l=p$ or $l=q$. 
Proceeding as in the proof of Lemma \ref{Proposition 5.1} with $m = \tau$ and $M = 0$, in accordance with \eqref{5.3KS}, we get
\begin{equation*}
n\mu(D_{2^{-(n+1)}\tau,  R})^{\frac{sl}{l-s}}\leq C\mu(B(y,R))^{\frac{sl}{l-s}}, 
\end{equation*} 
Independently of $l$, we can choose $n$ large enough so that
$$\mu(D_{2^{-(n+1)}\tau,  R})\leq \gamma_0 \mu(B(y,R)),$$
where $\gamma_0$ is as in Lemma \ref{Lemma6.1KS}.  Thus, using Lemma  \ref{Lemma6.1KS}, where $\tau$ is replaced by $2^{-(n+1)}\tau$, we conclude
$$\essinf_{B\left(y,\frac{R}{2}\right)} u \geq 2^{-(n+2)} \tau.$$
\end{proof}
 \begin{remark}[\cite{KS}, Remark 6.3]\label{Remark 6.3 KS}
 Let $0 < R \leq \frac{{\rm diam}(X)}{18}$ with $B(y, 6R) \subset \Omega$. We suppose that there exists $0 < \delta < 1$ such that
 $$\mu(\{x \in B(y,R): u(x) \geq \tau\}) \geq \delta \mu(B(y, R)).$$
Since the measure $\mu$ is doubling $($see \eqref{doubling}$)$, we have
 
 \begin{align*}\mu(\{x \in B(y,R): u(x) \geq \tau\}) &\geq \delta \mu(B(y, R))\geq \frac{\delta}{C_d}\mu( B(y,2R))\\&\geq \frac{\delta}{C_d^2}\mu( B(y,4R))\geq \frac{\delta}{C_d^2}\mu( B(y,3R))\\&\geq \frac{\delta}{C_d^3}\mu( B(y,6R))
 \end{align*}
 and so,
 $$\mu(\{x \in B(y,6R): u(x) \geq \tau\})\geq \mu(\{x \in B(y,R): u(x) \geq \tau\}) \geq \frac{\delta}{C_d^3}\mu( B(y,6R)),$$ where $C_d \geq 1$ is the doubling constant of $\mu$. Hence by Lemma \ref{Lemma6.2KS} we have 
 \begin{equation}\label{(6.1)KS}
 \essinf_{B\left(y,3R\right)} u \geq \lambda \tau.
 \end{equation}
 where $\lambda > 0$ is as in Lemma \ref{Lemma6.2KS}.
Clearly we may assume that $0 < \lambda < 1$.
 \end{remark}

Now, we report the Krylov-Safonov covering Theorem \cite{KrS} on a doubling
metric measure space which has been proven in \cite{KS}.
\begin{lemma}[\cite{KS}, Lemma 7.2]\label{Lemma 7.2KS}
Let $B(y, R)$ be a ball in $X$, and $E \subset B(y, R)$ be $\mu$-measurable. Let $0 < \delta < 1$, and define
\begin{equation}\label{Edelta}
E_\delta = \cup_{\rho>0}\{B(x, 3\rho) \cap B(y, R): x\in B(y, R), \mu(E\cap B(x, 3\rho))> \delta \mu( B(x, \rho)) \}.
\end{equation}
Then, either $E_\delta = B(y, R)$, or else $\mu(E_\delta) \geq (C_d\, \delta)^
{-1}\mu(E)$, where $C_d\geq 1$  is the doubling constant of $\mu$.
\end{lemma}
At this point, we are ready to state the following theorem.
\begin{theorem}\label{Theorem 7.1KS}
We suppose that $u > 0$ and $-u$ satisfies \eqref{5.4}. Then there exist two constants 
$C$ and $\sigma> 0$  such that
\begin{equation}
\essinf_{B(y, 3R)} u \geq C \left(\dashint_{B(y, R)}u^{\sigma} \, d\mu\right)^{\frac{1}{\sigma}}
\end{equation}
for every $B(y, R)$ with $B(y, 6R) \subset \Omega$ and $0 < R \leq \frac{R(y)}{6}$. The
constants $C$ and $\sigma$ are independent of the ball $B(y,R)$.
\end{theorem}
\begin{proof}
Let $0 < \delta,  \lambda< 1$ given as in Remark \ref{Remark 6.3 KS}. We define
$$A_{t,j}=\{x \in B(y,R): u(x) \geq t \lambda^j\}, \quad \mbox{where $t > 0$ and $j\in \mathbb{N}\cup \{0\}$}.$$ We apply Krylov-Safonov covering Theorem (Lemma \ref{Lemma 7.2KS}) with $E = A_{t,j-1}$. If there is a point $x \in B(y, R)$ and $\rho > 0$ so that $$\mu(A_{t,j-1} \cap B(x, 3\rho))\geq \delta \mu(B(x, \rho))),$$
then, by the doubling property, $$\mu(A_{t,j-1} \cap B(x, 6\rho))\geq \frac{\delta}{C_d^3} \mu(B(x, 6\rho))),$$
and, by Remark \ref{Remark 6.3 KS}, we have
$$\essinf_{B(x, 3\rho))} u \geq t \lambda^j.$$
Hence if $B(x, 3\rho)$ is as in \eqref{Edelta}, then $B(x, 3\rho)\cap B(y, R) \subset A_{t,j}$. This implies that $E_{\delta} \subset A_{t,j}$, consequently we have that $\mu(E_{\delta}) \leq \mu(A_{t,j})$. Hence we conclude that 
\begin{equation}\label{7.2KS}
\frac{1}{K_d\, \delta} \mu(A_{t,j-1}) \leq \mu(E_{\delta}) \leq \mu(A_{t,j})
\end{equation}or $E_{\delta}= B(y, R)$ and so, $A_{t,j}= B(y, R)$.
 
 We note that, since $u \in N^{1,q}_{loc}(\Omega)\subset L^q(\Omega)\subset L^1(\Omega)$, we have that $$\lim_{t \to +\infty} \mu(A_{t,0})=0.$$
If $A_{t,0} =B(y,R)$ for some $t>0$, then there exists $t_0 >0$ such that  \begin{equation}\label{Ato}
\mu(A_{t,0}) < \mu(B(y, R)) \quad\mbox{for all $t > t_0$} 
\end{equation} and  $\mu(A_{t,0}) = \mu(B(y, R))$ for all $t < t_0$. Let $0<\delta< \frac{1}{C_d}$. Now, if $t > t_0$, \eqref{Ato} implies that $$\frac{\mu(A_{t,0})}{\mu(B(y, R))} <1,$$ so we can choose an integer $n \geq 1$ such that $$(C_d \delta)^j \leq \frac{\mu(A_{t,0})}{\mu(B(y, R))} \leq (C_d \delta)^{n-1}.$$
Using \eqref{7.2KS}, we get the following chain of inequalities
\begin{equation*}
\mu(A_{t,n-1}) \geq \frac{1}{C_d \delta} \mu(A_{t,n-2}) \geq ... \geq \frac{1}{(C_d \delta)^{n-1}} \mu(A_{t,0}) \geq C_d \delta \mu(B(y, R)).
\end{equation*}
By Remark \ref{Remark 6.3 KS}, we see that 
\begin{equation*}
\essinf_{B(y, 3R)} u \geq C t \lambda^{n-1}.
\end{equation*}
This implies that
$$t_0= \essinf_{B(y, R)}u\geq \essinf_{B(y, 3R)} u \geq C t \lambda^{n-1}=C t(C_d \delta)^{\frac{(n-1)(\log \lambda)}{\log (C_d \delta)}}\geq C t \left(\frac{\mu(A_{t,0})}{\mu(B(y, R))}\right)^{\gamma},$$
where $\gamma= \frac{\log \lambda}{\log (C_d \delta)}$. Consequently we obtain the estimate
\begin{equation*}
\frac{\mu(A_{t,0})}{\mu(B(y, R))} \leq C t^{-\frac{1}{\gamma}} t_0^{\frac{1}{\gamma}}.
\end{equation*}
On the other hand, for $\sigma > 0$, we compute
\begin{align*}
\dashint_{\mu(B(y, R))} u^{\sigma} \, d\mu & = \frac{\sigma}{\mu(B(y, R))} \int_{0}^{+\infty} t^{\sigma -1} \mu(A_{t,0})\, dt \\ &\leq \frac{\sigma}{\mu(B(y, R))} \int_{t_0}^{+\infty} t^{\sigma -1} \mu(A_{t,0})\, dt + \sigma \int_0^{t_0} t^{\sigma -1} \, dt\\
&\leq C \int_{t_0}^{+\infty} t^{\sigma -1- \frac{1}{\gamma}}t_0^{\frac{1}{\gamma}} \, dt + t_0^{\sigma}.
\end{align*} If $\sigma < \frac{1}{\gamma}$, then 
\begin{equation*}
\dashint_{\mu(B(y, R))} u^{\sigma} \, d\mu \leq C t_0^{\frac{1}{\gamma}} \left(\frac{1}{\gamma}- \sigma\right)^{-1} t_0 ^{\sigma-\frac{1}{\gamma}}+ t_0^{\sigma} \leq C t_0^{\sigma},
\end{equation*}
and hence,
\begin{equation*}
t_0 \geq C \left(\dashint_{\mu(B(y, R))} u^{\sigma} \, d\mu \right)^{\frac{1}{\sigma}}.
\end{equation*} 

This completes the proof.
\end{proof}

Combining Theorem \ref{Theorem 4.2} with $k_0=0$ and Theorem \ref{Theorem 7.1KS} we deduce the Harnack inequality for $(p,q)$-quasiminimizers, thus Theorem \ref{Corollary 7.3 KS} is proven. 

In the particular case in which $u$ is a $(p,q)$-harmonic function, then Theorem \ref{Corollary 7.3 KS} becomes the next corollary.
\begin{corollary}\label{Corollary 7.3contin KS}
 Suppose that  $u$ is continuous, $u > 0$, $u$ and $-u$ satisfy \eqref{5.4}. Then there
exists a constant $C \geq 1$ so that
$$\sup_{B(y,R)} u \leq C \inf_{B(y,R)} u,$$
for every ball $B(y,R)$ for which $B(y, 6R) \subset \Omega$ and $0 < R \leq \frac{R(y)}{6}$. 
\end{corollary}
Here, we give two consequences of the Harnack inequality of Corollary \ref{Corollary 7.3contin KS}. Firstly, following the lines of Theorem 8.13 of \cite{BB}, we obtain the strong maximum principle for $(p,q)$-harmonic functions.
\begin{corollary}[Strong maximum principle] If $\Omega$ is connected, $u$ is a $(p,q)$-harmonic function
in $\Omega$ and $u$ attains its maximum in $\Omega$, then $u$ is constant in $\Omega$.
\end{corollary}
\begin{proof} It is not restrictive to suppose that $\max_{\Omega} u = 0$.  We consider $$A=\{x \in \Omega : u(x)=0\}.$$ The continuity of $u$ implies that $A$ is a relatively closed subset of $\Omega$. Let $x_0 \in A$ and $x_0 \in B(y, R)$ such that  $B(y, 6 R) \subset \Omega$. We remark that $-u$ is a nonnegative $(p,q)$-quasiminimizer in $\Omega$ and that, by the hypothesis of continuity on $u$, $$\essinf_{B(y, R)} u=\inf_{B(y, R)} u \quad \mbox{and} \quad \esssup_{B(y, R)} u=\sup_{B(y, R)} u.$$ Thus, we can apply Harnack inequality (Corollary \ref{Corollary 7.3contin KS}) to $-u$ and get the following inequality.
\begin{equation*}\label{-inf}
-\inf_{B(y, R)} u= \sup_{B(y, R)} (-u) \leq  C \inf_{B(y, R)} (-u)= -C \sup_{B(y, R)} u=0,
\end{equation*}
and so, $$0 \leq \inf_{B(y, R)} u \leq \sup_{B(y, R)} u=0.$$ Hence $B(y, R) \subset A$, that is $A$ is open. Since $\Omega$  is connected, we get that $A=\Omega$.  Thus, $u(x)=0$ for all $x \in \Omega$, so $u$ is constant in $\Omega$.
\end{proof}

The second corollary of Harnack inequality is the Liouville's Theorem.
\begin{corollary}[Liouville's Theorem]\label{Corollary 8.16contin BB}
If $u$ is a $(p,q)$-harmonic function bounded from below in $\Omega$, then $u$ is constant.
\end{corollary}
\begin{proof}
Let $v = u- \inf_{\Omega} u \geq 0$.
We note that, since $u$ is a $(p,q)$-harmonic function, $v$ is continuous itself. Thus, $$\essinf_{B(y, R)} v=\inf_{B(y, R)} v \quad \mbox{and} \quad \esssup_{B(y, R)} v=\sup_{B(y, R)} v.$$ 
Now, we can apply Harnack inequality (Corollary \ref{Corollary 7.3contin KS}) to get that, for $x \in \Omega$, 
$$v(x)\leq \sup_{B(x,\rho)}v \leq C\inf_{B(x,\rho)}v \to 0, \quad \mbox{as $\rho \to +\infty$.}$$
Thus $v \equiv 0$, so $u$ is constant.
\end{proof}

\section{Boundary continuity for $(p,q)$-quasiminimizers} \label{Sec6}
Up to this point, we have been interested in local properties, now we will focus our attention in boundary value problems. Thus, we have the need to introduce a boundary data $w$. 
More specifically, we consider a bounded domain $\Omega \subset X$ and $w \in N^{1,q}(X)$. Let $u \in N^{1,q}_{loc}(\Omega)$ be a $(p,q)$-quasiminimizer on $\Omega$ such that $w - u \in N^{1,q}_0(\Omega) $. The function $u$ is extended to all of $X$ by setting $u= w$ in $X \setminus \Omega$.

The following proposition, which is used in the proof of Proposition \ref{prop4.4},  is a consequence of the weak Poincar\'{e} inequality. It was first proven by Maz'ya, \cite{Maz} for the Euclidean setting, then generalized to the metric setting by Bj\"{o}rn, \cite{B}. We refer the reader to the corresponding references for more details. 
\begin{proposition}\label{Proposition3.2}
	Let $u\in N^{1,q}(X)$, $B(y,r)\subset X$ and $S=\{x\in B\left(y,\frac{r}{2}\right) : u(x)=0\}$. Then there exist $C$ and $\lambda \geq 1$ such that $$\dashint_{B(y,r)}|u|^qd\mu \leq \frac{C}{{{\rm cap}}_q(S,B(y,r))}\int_{B(y,\lambda r)}g_u^q\, d\mu.$$
\end{proposition}

Let $w\in N^{1,q}(X)$ such that $u-w\in N^{1,q}_{0}(\Omega)$.
We shall use the notation
\begin{equation*}
M(r,r_{0})=\Big(\esssup_{B(x_{0},r)} u -\esssup_{B(x_{0},r_{0})} w\Big)_{+},
\end{equation*}
where $0<r\leq r_{0}$, $a_{+}=\max\lbrace a,0\rbrace$, and $u\in N^{1,q}(X)$ is a $(p,q)$-quasiminimizer in $\Omega$. Let also
\begin{equation}
\gamma(s,r)=\frac{r^{-s}\mu(B(x_{0},r))}{{\rm cap}_{s}(B(x_{0},r)\setminus\Omega, B(x_{0},2r))}.
\end{equation}
The next proposition is important in order to prove Theorem \ref{Theorem2.11}, which gives a pointwise estimate for $(p,q)$-quasiminimizers near a boundary point. 

\begin{proposition}\label{prop4.4} Let $u\in N^{1,q}(X)$ be a $(p,q)$-quasiminimizer on $\Omega$ and $w\in N^{1,q}(X)$ with $u-w\in N^{1,q}_{0}(\Omega)$. Then there exist $C=C(\alpha, \beta, C_d, C_{PI}, K, p, Q)$ and $\lambda\geq 1$ such that, for all $x_{0}\in\partial\Omega$ and $0<2\lambda r\leq r_{0}<R(y)$, the next inequality holds
\begin{equation*}
M\left(\frac{r}{2}, r_{0}\right)\leq (1-2^{-n(r)-1})M(2\lambda r, r_{0}),
\end{equation*}
where
\begin{equation*}
n(r)=C\gamma\left(s,\frac{r}{2}\right)^{\frac{p}{p-s}}.
\end{equation*}
\end{proposition}

\begin{proof}
We denote $M=M(2\lambda r, r_{0})$, where $x_{0}$ and $r_{0}$ are fixed. Without loss of generality, we can assume that $0< M < +\infty$, otherwise the proof is finished. We define
\begin{align*}
k_{j}&=\esssup_{B(x_{0},r_{0})} w+M(1-2^{-j}),\\
v_{j}&=(u-k_{j})_{+}-(u-k_{j+1})_{+}.
\end{align*}
Notice that in $B(x_{0},\lambda r)\setminus\Omega$ we have $v_{j}=0$, where $\lambda$ is given by  Proposition \ref{Proposition3.2}. Define $T(k,l,r)=S_{k,r}\setminus S_{l,r}$, then $g_{u}\chi_{T(k_{j},k_{j+1},\lambda r)}$ is a $p$-weak upper gradient of $v_{j}$ in $B(x_{0},\lambda r)$. As a consequence of doubling property of $\mu$, H\"older inequality and Proposition \ref{Proposition3.2} with $s$ instead of $q$, we get that
\begin{align*}
\int_{B(x_{0},r)}v_{j}^{s} \, d\mu&\leq\frac{C\mu(B(x_{0},r))}{{\rm cap}_{s}\left(B\left(x_{0},\frac{r}{2}\right)\setminus\Omega, B(x_{0},r)\right)}\int_{B(x_{0},\lambda r)}g_{v_{j}}^{s}\, d\mu\\
&= C\gamma\left(s,\frac{r}{2} \right) r^{s}\int_{S_{k_{j},\lambda r}}g_{u}^{s}\chi_{T(k_{j},k_{j+1},\lambda r)}^{s}\, d\mu\\
&\leq C\gamma\left(s,\frac{r}{2}\right) r^{s}\left(\int_{S_{k_{j},\lambda r}}g_{u}^{l} \, d\mu\right)^{\frac{s}{l}}\mu(T(k_{j},k_{j+1},\lambda r))^{1-\frac{s}{l}},
\end{align*}
with $C$ depending on $C_d$, and therefore,
\begin{equation}\label{eq1prop4.4}
\int_{B(x_{0},r)}v_{j}^{s}\,d\mu\leq C\gamma\left(s,\frac{r}{2}\right)r^{s}\left(\int_{S_{k_{j},\lambda r}}g_{u}^{l}\, d\mu\right)^{\frac{s}{l}}\mu(T(k_{j},k_{j+1},\lambda r))^{1-\frac{s}{l}},
\end{equation}
for $l=p,q$.

For the measure of the set $S_{k_{j+1}, r}$, we have the following estimate.
\begin{equation*}
\int_{B(x_{0},r)}v_{j}^{s}\, d\mu\geq (k_{j+1}-k_{j})^{s}\mu(S_{k_{j+1}, r})=\frac{M^{s}}{2^{s(j+1)}}\mu(S_{k_{j+1}, r}).
\end{equation*}
By the De Giorgi type inequality, H\"older inequality and the doubling property, we have 
\begin{align*}
\int_{S_{k_{j},\lambda r}}g_{u}^{l}\, d\mu&\leq\int_{S_{k_{j},\lambda r}}(ag_{u}^{p}+bg_{u}^{q})\, d\mu\\ &\leq\frac{C}{(\lambda r)^{p}}\int_{S_{k_{j},2\lambda r}}a(u-k_{j})^{p}_{+}\, d\mu+\frac{C}{(\lambda r)^{q}}\int_{S_{k_{j},2\lambda r}}b(u-k_{j})^{q}_{+}\, d\mu\\ 
&\leq\frac{C}{(\lambda r)^{p}}\int_{S_{k_{j},2\lambda r}}(u-k_{j})^{p}_{+}\, d\mu+\frac{C}{(\lambda r)^{q}}\int_{S_{k_{j},2\lambda r}}(u-k_{j})^{q}_{+}\, d\mu\\ 
&\leq \frac{C}{(\lambda r)^{p}}\Big(\esssup_{B(x_{0},2\lambda r)}u-k_{j}\Big)^{p}\mu(B(x_{0},2\lambda r))\\ &\quad+\frac{C}{(\lambda r)^{q}}\Big(\esssup_{B(x_{0},2\lambda r)}u-k_{j}\Big)^{q}\mu(B(x_{0},2\lambda r))\\ 
&\leq \frac{C}{(\lambda r)^{p}}\frac{M^{p}}{2^{jp}}\mu(B(x_{0},2\lambda r))+\frac{C}{(\lambda r)^{q}}\frac{M^{q}}{2^{jq}}\mu(B(x_{0},2\lambda r))\\ &\leq \frac{C}{r^{p}}\frac{M^{p}}{2^{jp}}\mu(B(x_{0}, r))+\frac{C}{r^{q}}\frac{M^{q}}{2^{jq}}\mu(B(x_{0}, r))\\
&\leq C\max\left\{\left(\frac{M}{r2^{j}}\right)^{p}, \left(\frac{M}{r2^{j}}\right)^{q}\right\}\mu(B(x_{0}, r)), 
\end{align*}
for $l=p,q$, where $C$ now depends also on $\beta$, $K$ and $p$. Thus,
\begin{equation}\label{eq2prop4.4}
\int_{S_{k_{j},\lambda r}}g_{u}^{l}\, d\mu\leq C\max\left\lbrace\left(\frac{M}{r2^{j}}\right)^{p}, \left(\frac{M}{r2^{j}}\right)^{q}\right\rbrace\mu(B(x_{0}, r)), \quad \mbox{for $l=p,q$.}
\end{equation}
Using the last three inequalities, we get the following two cases.\\
\noindent \underline{Case 1:} $\max\lbrace\left(\frac{M}{r2^{j}}\right)^{p}, \left(\frac{M}{r2^{j}}\right)^{q}\rbrace=\left(\frac{M}{r2^{j}}\right)^{p}$. \\
Then, using $l=p$ in (\ref{eq1prop4.4}) and (\ref{eq2prop4.4}):
\begin{align*}
\frac{\mu(S_{k_{j+1},r})}{\mu(B(x_{0}, r))}&\leq\frac{2^{s(j+1)}}{\mu(B(x_{0}, r))M^{s}}\int_{B(x_{0}, r)}v_{j}^{s}\, d\mu\\
&\leq\frac{2^{s(j+1)}C}{\mu(B(x_{0}, r))M^{s}}  \gamma\left(s,\frac{r}{2}\right) r^{s}\left(\int_{S_{k_{j},\lambda r}}g_{u}^{p}\, d \mu\right)^{\frac{s}{p}}\mu(T(k_{j},k_{j+1},\lambda r))^{1-\frac{s}{p}}\\
&\leq\frac{2^{s(j+1)}C}{\mu(B(x_{0}, r))M^{s}}\gamma\left(s,\frac{r}{2}\right) r^{s}\left(\frac{M}{r2^{j}}\right)^{s}\mu(B(x_{0}, r))^{\frac{s}{p}} \mu(T(k_{j},k_{j+1},\lambda r))^{1-\frac{s}{p}}\\
&=C \gamma\left(s,\frac{r}{2}\right)\left(\frac{\mu(T(k_{j},k_{j+1},\lambda r))}{\mu(B(x_{0}, r))}\right)^{1-\frac{s}{p}}.
\end{align*}
Therefore,
\begin{equation}\label{eq3prop4.4}
\frac{\mu(S_{k_{j+1},r})}{\mu(B(x_{0}, r))}\leq C\gamma\left(s,\frac{r}{2}\right)\left(\frac{\mu(T(k_{j},k_{j+1},\lambda r))}{\mu(B(x_{0}, r))}\right)^{1-\frac{s}{p}}.
\end{equation}
\underline{Case 2:} $\max\lbrace\left(\frac{M}{r2^{j}}\right)^{p}, \left(\frac{M}{r2^{j}}\right)^{q}\rbrace=\left(\frac{M}{r2^{j}}\right)^{q}$.\\
Analogously, making $l=q$ in (\ref{eq1prop4.4}) and (\ref{eq2prop4.4}) we get
\begin{equation}\label{eq4prop4.4}
\frac{\mu(S_{k_{j+1},r})}{\mu(B(x_{0}, r))}\leq C\gamma \left(s,\frac{r}{2}\right)\left(\frac{\mu(T(k_{j},k_{j+1},\lambda r))}{\mu(B(x_{0}, r))}\right)^{1-\frac{s}{q}}.
\end{equation}
If $n\geq j+1$, notice that both (\ref{eq3prop4.4}) and (\ref{eq4prop4.4}) hold true when the set $S_{k_{j+1},r}$ on the left part of the inequalities is replaced by $S_{k_n,r}$. Thus, for $l=p,q$,
\begin{equation*}
\left(\frac{\mu(S_{k_n,r})}{\mu(B(x_{0}, r))}\right)^{\frac{l}{l-s}}\leq C\gamma\left(s,\frac{r}{2}\right)^{\frac{l}{l-s}}\frac{\mu(T(k_{j},k_{j+1},\lambda r))}{\mu(B(x_{0}, r))},
\end{equation*}
and, summing up over $j=0,1,...,n-1$, 
\begin{equation*}
\left(\frac{\mu(S_{k_n,r})}{\mu(B(x_{0}, r))}\right)^{\frac{l}{l-s}}\leq \frac{C}{n} \gamma\left(s,\frac{r}{2}\right)^{\frac{l}{l-s}}.
\end{equation*}
As a consequence,  for $l=p,q$, we get
\begin{equation*}
\frac{\mu(S_{k_n,r})}{\mu(B(x_{0}, r))}\leq \frac{C}{n^{\frac{l}{l-s}}}\gamma\left(s,\frac{r}{2}\right).
\end{equation*}
Notice that independently of the case we have
\begin{equation}\label{eq5prop4.4}
\frac{\mu(S_{k_n,r})}{\mu(B(x_{0}, r))}\leq \frac{C}{n^{\frac{p}{p-s}}} \gamma\left(s,\frac{r}{2}\right).
\end{equation}
By Theorem \ref{Theorem 4.2} with $k_{n}$ and $r$ instead of $k_{0}$ and $R$, we obtain
\begin{align*}
\esssup_{B(x_{0},\frac{r}{2})} u&\leq k_{n}+C\left(\dashint_{B(x_{0}, r)}(u-k_{n})_{+}^{q}\, d\mu\right)^{\frac{1}{q}}\\
&\leq \esssup_{B(x_{0},\frac{r}{2})} w+M(1-2^{-n})+C\left(\dashint_{B(x_{0}, r)}(u-k_{n})_{+}^{q}\, d \mu\right)^{\frac{1}{q}}\\
&\leq \esssup_{B(x_{0},\frac{r}{2})} w+M(1-2^{-n})\\&\quad+C\left(\frac{1}{\mu(B(x_{0}, r))}\int_{S_{k_{n},r}}\big(\esssup_{B(x_{0},2\lambda r)}u-k_{n}\big)^{q}\, d\mu\right)^{\frac{1}{q}}\\
&\leq \esssup_{B(x_{0},\frac{r}{2})} w+M(1-2^{-n})+C\left(\frac{1}{\mu(B(x_{0}, r))}\int_{S_{k_{n},r}}(M^{q}2^{-nq})\, d\mu\right)^{\frac{1}{q}}\\
&=\esssup_{B(x_{0},\frac{r}{2})} w+M(1-2^{-n})+CM2^{-n}\left(\frac{\mu(S_{k_{n},r})}{\mu(B(x_{0}, r))}\right)^{\frac{1}{q}},
\end{align*} where, clearly, $C$ has the same dependencies of the constant of Theorem \ref{Theorem 4.2}, that is $C=C(\alpha, \beta, C_{PI}, Q)$.
By (\ref{eq5prop4.4}) we then have
\begin{equation*}
\esssup_{B(x_{0},\frac{r}{2})} u \leq\esssup_{B(x_{0},\frac{r}{2})} v+M(1-2^{-n})+\frac{CM}{2^{n}n^{(1-\frac{s}{p})\frac{1}{q}}}\gamma \left(s,\frac{r}{2}\right)^{\frac{1}{q}}.
\end{equation*}
If $n\geq C\gamma \left(s,\frac{r}{2}\right)^{\frac{p}{p-s}}$ then, the last term on the right-hand side of the last inequality is at most $2^{-n-1}M$. So,
\begin{equation*}
\esssup_{B(x_{0},\frac{r}{2})} u \leq \esssup_{B(x_{0},\frac{r}{2})} v+M(1-2^{-n})+2^{-n-1}M.
\end{equation*}
Finally, we get
\begin{equation*}
M\left(\frac{r}{2},r_{0}\right)\leq(1-2^{-n-1})M.
\end{equation*}
At last, by choosing the smallest integer $n(r)$ for which $n\geq n(r) \geq C\gamma \left(s,\frac{r}{2}\right)^{\frac{p}{p-s}}$, the proof is completed. 
\end{proof}

Now we are ready to prove the next theorem. This is a pointwise estimate for $(p,q)$-quasiminimizers near a boundary point, see for example \cite{B, Maz}.

\begin{theorem}\label{Theorem2.11}
Let $u\in N^{1,q}(X)$ be a $(p,q)$-quasiminimizer on $\Omega$ and $w\in N^{1,q}(X)$ with $u-w\in N^{1,q}_{0}(\Omega)$. Then there exist $C_{0},C_{1}>0$ such that
\begin{equation}
M(\rho,r_{0})\leq C_{1}M(r_{0},r_{0})\exp\left(-\frac{1}{4}\int_{\rho}^{r_{0}}\exp\left(-C_{0}\gamma(s,r)^{\frac{p}{p-s}}\right)\frac{\, d r}{r}\right)
\end{equation}
\end{theorem}

\begin{proof}
	We denote $M(r)=M(r,r_{0})$, where $r_{0}$ is fixed. Without loss of generality, we can assume that $0< M(r_{0}) < +\infty$, otherwise the proof is finished. 
 Let $C$ and $n(r)$ be as in Proposition \ref{prop4.4}. We define $C_{0}=C\log 2$ and
\begin{equation*}
\omega(r)=\exp\left(-C_{0}\gamma(s,r)^{\frac{p}{p-s}}\right)=2^{-n(2r)}.
\end{equation*}
Now we consider $(0,r_{0})=I_{1}\cup I_{2}$ with $I_{1}\cap I_{2}=\emptyset$, where
\begin{equation*}
I_{m}=\bigcup_{j=1}^{+\infty}\left[(4\lambda)^{m-2j-1}r_{0}, (4\lambda)^{m-2j}r_{0}\right), \quad \mbox{for $m=1,2$.} 
\end{equation*}
Thus, 
\begin{equation}\label{eq1theo2.11}
\int_{\rho}^{r_{0}}\omega(r)\frac{dr}{r}\leq 2\int_{(\rho,r_{0})\cap I_{m}}\omega(r)\frac{dr}{r},
\end{equation} for $m=1$ or $m=2$. Now, for every $j \in \mathbb{N}$, let $r_{j}\in \left[(4\lambda)^{m-2j-1}r_{0}, (4\lambda)^{m-2j}r_{0}\right)$  be such that
\begin{equation*}
\omega(r_{j})\geq\frac{1}{(4\lambda)^{m-2j-1}r_{0}}\int_{(4\lambda)^{m-2j-1}r_{0}}^{(4\lambda)^{m-2j}r_{0}}\omega(r)dr\geq \int_{(4\lambda)^{m-2j-1}r_{0}}^{(4\lambda)^{m-2j}r_{0}}\omega(r)\frac{dr}{r}.
\end{equation*}
Since $\omega(r)\leq 1$ for all $r$, we then get
\begin{equation}\label{eq2theo2.11}
\int_{(\rho,r_{0})\cap I_{m}}\omega(r)\frac{dr}{r}\leq\sum_{\rho\leq r_{j}<r_{0}/4\lambda}\omega(r_{j})+C.
\end{equation}
Using Proposition \ref{prop4.4}, we deduce
\begin{align*}
M((4\lambda)^{m-2j-1}r_{0})&\leq M(r_{j})\leq M(4\lambda r_{j})(1-2^{-n(2r_{j})-1})\\
&\leq M(4\lambda)^{m-2j+1}r_{0}\left(1-\frac{\omega(r_{j})}{2}\right), \quad \mbox{for $j\in \mathbb{N}$.}
\end{align*}
Proceeding by iteration and recalling that $\log(1-t)\leq -t$, we get
\begin{equation*}
M(\rho)\leq M(r_{0})\exp\left(-\frac{1}{2}\sum_{\rho\leq r_{j}<r_{0}/4\lambda}\omega(r_{j})\right),  \quad \mbox{for $0<\rho<r_{0}$.}
\end{equation*}
Now, notice that we have
\begin{equation*}
-\frac{1}{2}\sum_{\rho\leq r_{j}<\frac{r_{0}}{4\lambda}}\omega(r_{j})\leq C-\frac{1}{2}\int_{(\rho,r_{0})\cap I_{m}}\omega(r)\frac{dr}{r}\leq C-\frac{1}{4}\int_{\rho}^{r_{0}}\omega(r)\frac{dr}{r}.
\end{equation*}
Then,
\begin{equation*}
\exp\left(-\frac{1}{2}\sum_{\rho\leq r_{j}<\frac{r_{0}}{4\lambda}}\omega(r_{j})\right)\leq C_{1}\exp\left(\frac{1}{4}\int_{\rho}^{r_{0}}\exp\left(-C_{0}\gamma(s,r)^{\frac{p}{p-s}}\right)\frac{dr}{r}\right)
\end{equation*}
as wanted.
\end{proof}
Theorem  \ref{Theorem2.11} implies the following result, which gives us a sufficient condition for the H\"older continuity of $(p,q)$-quasiminimizers at a boundary point.

\begin{theorem}\label{Theorem2.12}
Let $u\in N^{1,q}(X)$ be a $(p,q)$-quasiminimizer on $\Omega$ and $w\in N^{1,q}(X)$ be a H\"older continuous function at $x_{0}\in\partial\Omega$, with $u-w\in N^{1,q}_{0}(\Omega)$.  Then there exists $C_0>0$ such that
\begin{equation*}
\liminf_{\rho\rightarrow 0}\frac{1}{\vert \log \rho\vert}\int_{\rho}^{1}\exp\left(-C_0 \gamma(s,r)^{\frac{p}{p-s}}\right)\frac{dr}{r}>0.
\end{equation*}
 Thus $u$ is H\"older continuous at $x_{0}$.
\end{theorem}

\begin{proof}
By Remark \ref{Proposition 3.3KS}, $-u$ is a $(p,q)$-quasiminimizer. We notice that $-u-(-w) \in N^{1,q}_{0}(\Omega)$ by assumption, so it is enough to estimate $(u(x)-w(x_{0}))_{+}$. Without loss of generality, we assume that $w(x_{0})=0$. By the continuity of $w$ at $x_{0}$, Theorem \ref{Theorem 4.2} implies
\begin{equation}\label{eq1theo2.12}
M(r_{0},r_{0})\leq M=\esssup_{B(x_{0},R)}u_{+}<+\infty,
\end{equation} for some $R>0$ and all $0<r_{0}<R$.

Using Theorem \ref{Theorem2.11} we get, for $0<\rho<r_{0}$, 
\begin{align}\label{eq2theo2.12}
 \nonumber \esssup_{B(x_{0},\rho)}u_{+}&\leq \esssup_{B(x_{0},r_{0})}w_{+}+M(\rho,r_{0})\\
&\leq \esssup_{B(x_{0},r_{0})}w_{+}+C_{1}M\exp\left(-\frac{1}{4}\int_{\rho}^{r_{0}}\exp\left(-C_{0}\gamma(s,r)^{\frac{p}{p-s}}\right)\frac{dr}{r}\right).
\end{align} 
The hypotheses imply that there are $C, h,k>0$ satisfying 
\begin{equation*}
\esssup_{B(x_{0},r_{0})}w_{+}\leq Cr_{0}^{h}
\end{equation*}
and
\begin{equation*}
\int_{\rho}^{1}\exp\left(-C_{0}\gamma(s,r)^{\frac{p}{p-s}}\right)\frac{dr}{r}\geq k\vert\log\rho\vert,
\end{equation*}
for all sufficiently small $\rho$ and $r_{0}$.

Notice that
\begin{equation*}
\int_{x_{0}}^{1}\exp\left(-C_{0}\gamma(s,r)^{\frac{p}{p-s}}\right)\frac{dr}{r}\leq \int_{x_{0}}^{1}\frac{dr}{r}=\vert\log r_{0}\vert, 
\end{equation*}for all $0<r_{0}<1$. Thus, applying inequality (\ref{eq2theo2.12}), we get 
\begin{equation*}
\esssup_{B(x_{0},\rho)}u_{+}\leq Cr_{0}^{h}+C_{1}M\rho^{\frac{k}{4}}r_{0}^{-\frac{1}{4}},
\end{equation*}
for sufficiently small $\rho$ and $r_{0}$.

Lastly, we choose $r_{0}=\rho^{k '}$ with $0<k '<k$ to get that $u$ is H\"older continuous at $x_{0}$ up to redefining it on a measure zero set.
\end{proof}

In the following theorem, supposing that $w \in N^{1,q}(X)$ is continuous at $x_0\in \partial \Omega$, we show that $u$ is also continuous at $x_0$ under a Wiener type regularity condition.
	
\begin{theorem}
	
	Let $u \in N^{1,q}(X)$ be a $(p,q)$-quasiminimizer on $\Omega$ and $w\in N^{1,q}(X)$ be a continuous function at $x_{0}\in\partial\Omega$, with $u-w\in N^{1,q}_{0}(\Omega)$.
	 Then there exists $t>0$ such that 
$$\int_0^1 \left( \dfrac{{\rm cap}_s(B(x_0,r)\setminus \Omega, B(x_0, 2r))}{r^{-\frac{qs}{p}}\mu(B(x_0,r))}\right)^{t} \dfrac{dr}{r}=+\infty,$$ and therefore, $u$ is continuous at $x_0$.
\end{theorem}
	
\begin{proof}
By absurd, we suppose that $u$ is not continuous at $x_0$. Eventually considering  $-u$ instead of $u$, we can assume that $$\lim_{r \to 0} \esssup_{B(x_0,r)}u>w(x_0)=0.$$
Thus, there exist $k_0, M>0$ so that 
\begin{equation}\label{17B}
2k_0< \esssup_{B(x_0,r)}u<M, 
\end{equation}for all sufficiently small $r > 0$.

Let $\sigma$ be defined as in Theorem \ref{Theorem 7.1KS}, $0<\sigma'< \sigma$ and $0<\theta<1$. We define
\begin{equation*}
h(r)= \esssup_{B(x_0, 10 \lambda r)} u- \esssup_{B(x_0,r)} u \quad \mbox{and} \quad k(r)=\esssup_{B(x_0, 10 \lambda r)}u- h(r)^{\frac{\sigma'}{\sigma}}.
\end{equation*}
Since $\lim_{r \to 0} h(r)=0$, we get that $k(r)>k_0$ and 
\begin{equation}\label{18B}
h(r)^{\frac{\sigma'}{\sigma}}= \esssup_{B(x_0,10 \lambda r)} u-k(r) \leq h(r)^{\frac{\theta\sigma'}{\sigma}}, 
\end{equation}for all sufficiently small $r > 0$.

The continuity hypothesis of $w$ at $x_0$ implies that $\esssup_{B(x_0,r_0)} w<k_0$ for some $r_0>0$, thus $(u-k_0)_+=0$ a.e. on $B(x_0,r_0) \setminus \Omega$. Using Proposition \ref{Proposition3.2}, for sufficiently small $r > 0$, we get that
\begin{align}\label{19B}
	{\rm cap}_s(B(x_0,r)\setminus \Omega, B(x_0, 2r))  &\dashint_{B(x_0, 2r)} (u-k_0)^s_+ \, d\mu \leq C \int_{B(x_0, 2\lambda r)} g^s_{(u-k_0)_+} \, d\mu \nonumber \\ &\leq C \int_{S_{k(r), 2\lambda r}} g^s_u \, d\mu + C \int_{T(k_0, k(r), 2\lambda r)} g^s_u \, d\mu , 
\end{align}
where $S_{k,r}=\{x\in B(x_0, r): u(x)>k\}$ and $T(h,k,r)=S_{h,r} \setminus S_{k,r}$.\\
Now, we consider each one of the two integrals on the right-hand side in \eqref{19B} in order to estimate them.
From the H\"{o}lder inequality, Lemma \ref{lem 4.1} and \eqref{18B}, we get
\begin{align}\label{20B}
	\int_{S_{k(r), 2\lambda r}} g^s_u \, d\mu & \leq \left( \int_{S_{k(r), 2 \lambda r}} g^p_u \, d\mu\right)^{\frac{s}{p}} \mu\left(B(x_0, 2 \lambda r)\right)^{1-\frac{s}{p}}\nonumber \\
	& 
	 \leq \left( C\int_{S_{k(r), 2\lambda r}} (ag^p_u + bg^q_u) \, d\mu \right)^{\frac{s}{p}} \mu\left(B(x_0, 2 \lambda r)\right)^{1-\frac{s}{p}}\nonumber \\
		&\leq \Big(\frac{ C }{(8\lambda r)^{p}} \int_{S_{k(r), 10\lambda r}} a(u-k(r)) ^p d \mu	\nonumber\\ &\quad+\frac{C}{(8\lambda r)^{q}} \int_{S_{k(r), 10\lambda r}}b(u-k(r))^q d \mu\Big)^{\frac{s}{q}}\mu\left(B(x_0, 2 \lambda r)\right)^{1-\frac{s}{p}} \nonumber \\		
	&=\Big(\frac{ C}{(8\lambda r)^{p}} \int_{S_{k(r), 10\lambda r}} (u-k(r)) ^p d \mu	\nonumber\\ &\quad+\frac{C}{(8\lambda r)^{q}} \int_{S_{k(r), 10\lambda r}}(u-k(r))^q d \mu\Big)^{\frac{s}{p}}\mu\left(B(x_0, 2 \lambda r)\right)^{1-\frac{s}{p}} \nonumber \\	
		&\leq \left(\frac{ C}{r^{q}} \int_{S_{k(r), 10\lambda r}} (u-k(r)) ^p d \mu\right)^{\frac{s}{p}}\mu\left(B(x_0, 2 \lambda r)\right)^{1-\frac{s}{q}}\nonumber \\ 
			&\leq  \frac{ C \mu(B(x_0,r))}{r^{\frac{qs}{p}}}\left(\dashint_{B(x_0,10 \lambda r)} (u-k(r)) ^q_+ d \mu \right)^{\frac{s}{q}}\nonumber \\
		&\leq \frac{C\mu(B(x_0,r))}{ r^{\frac{qs}{p}}}h(r)^{\frac{s\theta \sigma'}{\sigma}},
\end{align}
where $C=C(\alpha, \beta)$.

Applying again H\"{o}lder inequality, we deduce
\begin{align}\label{21B}
\int_{T(k_0, k(r), 2\lambda r)} g^s_u \, d\mu \leq \left( \int_{S_{k_0, 2 \lambda r}} g^p_u \, d\mu\right)^{\frac{s}{p}} \mu\left(T(k_0, k(r), 2\lambda r)\right)^{1-\frac{s}{p}}.
\end{align}
Using Lemma \ref{lem 4.1} and \eqref{17B}, we can estimate the integral term on the right-hand side of inequality  \eqref{21B} as follows 
\begin{align}\label{22B}
	&\left(\int_{S_{k_0, 2\lambda r}} g^p_u \, d\mu\right)^{\frac{s}{p}} \leq \left(\int_{S_{k_0, 2\lambda r}} (ag^p_u+bg^q_u)\, d\mu\right)^{\frac{s}{p}}\nonumber \\
&\leq\left( \frac{C}{( 2\lambda r)^p} \int_{B(x_0,10 \lambda r)} a(u-k_0) ^p_+ d \mu+\frac{C}{(2\lambda r)^q}  \int_{B(x_0,10 \lambda r)}b(u-k_0)^q_+ d \mu\right)^{\frac{s}{p}}\nonumber \\	
	&\leq \frac{C}{ r^{\frac{qs}{p}}} \left( \int_{B(x_0,10 \lambda r)} a(u-k_0) ^p_+ d \mu\right)^{\frac{s}{p}}  +\frac{C}{ r^{\frac{qs}{p}}} \left( \int_{B(x_0,10 \lambda r)}b(u-k_0)^q_+ d \mu\right)^{\frac{s}{p}}\nonumber \\
		&\leq \frac{C}{ r^{\frac{qs}{p}}} \left( \int_{B(x_0,10 \lambda r)}(u-k_0) ^p_+ d \mu\right)^{\frac{s}{p}} +\frac{C}{ r^{\frac{qs}{p}}} \left( \int_{B(x_0,10 \lambda r)}(u-k_0)^q_+ d \mu\right)^{\frac{s}{p}}\nonumber \\
		&\leq \frac{C M^s \mu(B(x_0,r))^{\frac{s}{p}}}{ r^{\frac{qs}{p}}}.
\end{align}
Note that throughout the last inequalites, without of loss of generality, we have assumed $M<1$ and that we can choose $r$ sufficiently small. 

On the other hand, to estimate $ \mu\left(T(k_0, k(r), 2\lambda r)\right)$, we notice that $$v= \esssup_{B(x_0,10 \lambda r)} u - \max \{u, k_0\}$$ is a nonnegative function in $B(x_0,10 \lambda r)$.
Therefore, $(-v-k)_+ = (u-k')_+ $ with $k' =k+  \esssup_{B(x_0,10 \lambda r)} u$, for all $k \geq k^*=k_0-\esssup_{B(x_0,10 \lambda r)} u $. Now, for $k<k^*$ we get $-v-k>-v-k^* \geq 0$ on $B(x_0, 10 \lambda r)$.
 So, from Lemma \ref{lem 4.1} with $k'$ instead of $k$, we have that $-v$ satisfies \eqref{5.4} with $S_{k, \rho}=S_{k_0,10 \lambda r}$.
By Theorem \ref{Theorem 7.1KS}, we obtain that 
\begin{align}\label{This}
\nonumber  \mu(T(k_0, k(r), 2\lambda r))& \big(\esssup_{B(x_0,10 \lambda r)} u -k(r) \big)^{\sigma}\leq \int_{B(x_0,2\lambda r)} v^{\sigma} \, d\mu \nonumber\\&\leq C\mu(B(x_0,2\lambda r))\big(\esssup_{B(x_0,10 \lambda r)} u-\esssup_{B(x_0,2 \lambda r)} u \big)^{\sigma}.
\end{align}
Inequalities \eqref{This} and \eqref{18B} imply
\begin{equation}\label{the latter}
\mu(T(k_0, k(r), 2\lambda r)) \leq C \mu(B(x_0,r))\omega(r)^{\sigma-\sigma'}.
\end{equation}
Now, considering \eqref{the latter}, \eqref{20B}-\eqref{22B}, we return to inequality \eqref{19B} and get 
\begin{align}\label{23B}
	\dfrac{{\rm cap}_s(B(x_0,r)\setminus \Omega, B(x_0, 2r))}{r^{-\frac{qs}{p}}\mu(B(x_0,r))} &\left(\dashint_{B(x_0, 2r)} (u-k_0)^s_+ \, d\mu\right) \nonumber \\ & \leq C \left(\omega(r)^{\frac{s\theta \sigma'}{\sigma}}+ M^s \omega(r)^{(\sigma-\sigma')\left(1-\frac{s}{q}\right)}\right).
\end{align}
Furthermore, applying Theorem \ref{Theorem 4.2} and inequality \eqref{17B} we obtain, for sufficiently small $r > 0$,
\begin{align*}
k_0^q &\leq \big(\esssup_{B(x_0,r)} u-k_0\big)^q\leq C \dashint_{B(x_0, 2r)} (u-k_0)^q_+ \, d\mu \\ &\leq CM^{q-s}\dashint_{B(x_0, 2r)} (u-k_0)^s_+ \, d\mu.
\end{align*}
By choosing opportunely $\theta$ and $\sigma$ of \eqref{23B}, we deduce that, for sufficiently small $r > 0$ and all $t> \frac{1}{s}+\frac{q}{\sigma(q-s)}$,
$$\left(\dfrac{{\rm cap}_s(B(x_0,r)\setminus \Omega, B(x_0, 2r))}{r^{-\frac{qs}{p}}\mu(B(x_0,r))}\right)^{t} \leq C \omega(r).$$
Finally, we get the following inequality$$\int_0^1 \omega (r)\dfrac{dr}{r}  \leq \int_1^{10\lambda}\esssup_{B(x_0,r)} u \, \dfrac{dr}{r}<+\infty,$$ which ends the proof.
\end{proof}	
	
\section{A boundary regularity result for $(p,q)$-minimizers}\label{Sec7}
In a similar way to Section \ref{Sec6}, here we have the following assumptions. Let $\Omega \subset X$ be a bounded domain and let $w \in N^{1,q}(X)$. Let $u \in N^{1,q}(\Omega)$ be a $(p,q)$-minimizer on $\Omega$ such that $w - u \in N^{1,q}_0(\Omega) $. 
The function $u$ is extended to all of $X$ by setting $u = w$ in $X \setminus \Omega$.

In this section, we focus our attention to $(p,q)$-minimizers and we prove Theorem \ref{Theorem 5.1BMS} which gives us control over the oscillation of $(p,q)$-minimizer functions at boundary points. This result is motivated by Theorem 5.1 of \cite{BMS} for $p$-minimizers in the setting of metric measure spaces. A key role in the proof is played by the comparison principle, which unfortunately fails for quasiminimizers (for further informations see, for example, \cite{BJ}). However we were able to prove a comparison principle for $(p,q)$-minimizers, thus to show the following theorem, that is the main result of this section.
\begin{theorem}\label{Theorem 5.1BMS}
	Let $\Omega \subset X$ be a bounded domain and $w \in N^{1,q}(X) \cap C(\overline{\Omega})$.
	Consider a $(p,q)$-minimizer $u$ on $\Omega$ such that $w - u \in N^{1,q}_0(\Omega) $. If $x_0 \in \partial\Omega$ and $0 < \rho \leq r$, then 
	\begin{align}\label{Th5.1(1)BMS}
		{\rm osc}(u, \overline{\Omega}\cap B(x_0, \rho)) \leq &	{\rm osc}(w, \partial \Omega \cap B(x_0, 5r)) \nonumber \\&+{\rm osc}(w, \partial \Omega ) \, {\rm exp}\left(-C\int_\rho^r \varphi(x_0, X \setminus \Omega, t)^{\frac{1}{q-1}} \dfrac{dt}{t}\right),
	\end{align}
	for some constant $C$ and 
		$\varphi(x, E, r)=\dfrac{{\rm cap}_q(B(x,r)\cap E, B(x,2r))}{	{\rm cap}_q(B(x,r), B(x,2r))}.$
\end{theorem} We will obtain the proof of Theorem \ref{Theorem 5.1BMS} as a consequence of the next lemmata, among which there is the comparison principle for $(p,q)$-minimizers (see Lemma \ref{Theorem 6.4Sh}). These results were obtained by adapting the proof tecniques used in \cite{Sh}.
In particular, for the proof of the first lemma we refer the reader to Lemma 6.5 of \cite{Sh}.
\begin{lemma}\label{Lemma6.5}
	Let $\Omega \subset X$. If $u$ and $v$ are two functions in $N^{1,q}(\Omega)$ and $E$ is a subset of $\Omega$ such that $u \leq v$ $q$-q.e. on $E$ and $u>v$ $q$-q.e. in $\Omega \setminus E$, then
	$w=u \chi_E + v \chi_{\Omega \setminus E}$ is also in $N^{1,q}(\Omega)$.
\end{lemma}
Now, we prove the comparison principle for $(p,q)$-minimizers. From now on, we will denote with $N^{1,q}_0 (E)+w $ the set of all functions of the form $v + w$
with $v \in  N^{1,q}_0 (E)$.
Following the lines of Theorem 6.4 of  \cite{Sh}.
\begin{lemma}[Comparison principle]\label{Theorem 6.4Sh}
	Let $\Omega\subset X$. We consider $V \subset \Omega' \subset \overline{\Omega'} \subset  \Omega$ and $u_1, u_2\in N^{1,q}({\Omega})$  $(p,q)$-minimizers.  If $v \leq  u_1$
	$q$-q.e. in $\overline{\Omega'} \setminus V$, then $u_2 \leq u_1$ $q$-q.e. in $V$.
\end{lemma}
\begin{proof}
	Using Remark \ref{quasiunicity}, we get that $(p,q)$-minimizers on $\Omega$ are unique up to sets of zero $q$-capacity. We consider the following set $E=\{x \in V : u_2(x) > u_1(x)\}$. We shall prove that $\mu(E)=0$, that is the set $E$ has measure zero.
	Now, we define $\Gamma_0$ as the class of all compact rectifiable paths on which at least
	one of $u_1, u_2$ is not absolutely continuous or at least one of the function weak
	upper gradient pairs $(u_1, g_{u_1})$, $(u_2, g_{u_2})$ does not satisfy inequality \eqref{ug}. The $q$-modulus of $\Gamma_0$ is zero.\\
	We define a function $v$ in $V$ in the following way
	$$v=
	\begin{cases}
		u_1 \quad \mbox{in $\Omega \setminus E$,}\\
		u_2 \quad \mbox{in $E$.}
	\end{cases}$$
	From the hypothesis, we have that $u_2 \leq u_1$ in $\Omega' \setminus E$ and $u_2>u_1$ in $E$, and from Lemma \ref{Lemma6.5} the
	function $v$ is in $N^{1,q}(\Omega)$. Now, applying Lemma 3.2 and Lemma 3.4 of \cite{Sh}, the function
	$g_0=g_{u_2} \chi_E +g_{u_1} \chi_{\Omega \setminus E}$ is a weak upper gradient of $v$. Therefore, since $u_1$ is a $(p,q)$-minimizer and $v \in N_0^{1,q}(\Omega')+ u_1$, we get
	\begin{align*}\int_E a g_{u_2}^p d\mu + \int_E b g_{u_2}^q d\mu +\int_{\Omega' \setminus E} a g_{u_1}^p d\mu+\int_{\Omega' \setminus E} b g_{u_1}^q d\mu &= \int_{\Omega'}a  g_0^p d\mu+\int_{\Omega'}b g_0^q d\mu\\& \geq \int_{\Omega'} a g_{u_1}^p d\mu+\int_{\Omega'} b g_{u_1}^q d\mu.
	\end{align*}
	thus,
	$$\int_E a g_{u_2}^p d\mu + \int_E b g_{u_2}^q d\mu  \geq \int_E a g_{u_1}^p d\mu + \int_E b g_{u_1}^q d\mu.$$
	We define a function $f$ in $V$ as
	$$f=
	\begin{cases}
		u_1 \quad \mbox{in $E$,}\\
		u_2 \quad \mbox{in $V \setminus E$.}
	\end{cases}$$
	Again using Lemma 3.2 and Lemma 3.4 of \cite{Sh}, function
	$g_1=g_{u_1} \chi_E +g_{u_2} \chi_{V \setminus E}$ is a weak upper gradient of $f$. Therefore, since $u_2$ is a $(p,q)$-minimizer in $V$ and since $f \in N_0^{1,q}(V)+u_2$,
	\begin{align*}\int_Ea  g_{u_1}^p d\mu + \int_E b g_{u_1}^q d\mu +\int_{V \setminus E}a  g_{u_2}^p d\mu+\int_{V \setminus E} b g_{u_2}^q d\mu &= \int_{V}a g_1^p d\mu+\int_{V} b g_1^q d\mu\\& \geq \int_{V} a g_{u_2}^p d\mu+\int_{V} b g_{u_2}^q d\mu.
	\end{align*}
	that is, $$\int_E a g_{u_2}^p d\mu + \int_E b g_{u_2}^q d\mu \leq \int_E a g_{u_1}^p d\mu + \int_E b g_{u_1}^q d\mu.$$ 
	As a consequence, we deduce that $$\int_E a g_{u_2}^p d\mu + \int_E b g_{u_2}^q d\mu=\int_E a g_{u_1}^p d\mu + \int_E b g_{u_1}^q d\mu,$$
	and hence, since $g_v \leq g_0$ almost everywhere,
	$$\int_E ag_v^p d\mu + \int_E bg_v^q d\mu = \int_E ag_{u_1}^p d\mu + \int_E bg_{u_1}^q d\mu.$$
	Therefore we get that $v$ is also a $(p,q)$-minimizer in $\Omega$. Hence, by the uniqueness property of $(p,q)$-minimizers (Remark  \ref{quasiunicity}), $v=u_1$ $q$-q.e. in $\Omega$, and so $\mu(E)=0$.
\end{proof}
Motivated by Lemma 5.7 of \cite{BMS}, we get the following lemma.
\begin{lemma}\label{Lemma5.7BMS}
	Let $\Omega \subset X$ be a bounded domain and $x_0 \in \partial \Omega$. Fix $r>0$ and let $u$ be the $(p,q)$-potential $($see Definition \ref{(p,q)-potential}$)$ for $\overline{B(x_0,r)}\setminus \Omega$ with respect to $B(x_0,5r)$. Then for $0<\rho\leq r$ and $x \in B(x_0, \rho)$, we have $$1-u(x) \leq {\rm exp}\left(-C\int_\rho^r \varphi(x_0, X \setminus \Omega, t)^{\frac{1}{q-1}} \dfrac{dt}{t}\right),$$ for some constant $C$ and $\varphi$ defined as in Theorem \ref{Theorem 5.1BMS}.
\end{lemma}
\begin{proof}
	Let $B= B(x_0, 5r)$ and $B_j= 5^{-j}B$. Let $u_j$ be a $(p,q)$-potential for $\overline{B}_j\setminus \Omega$ with respect to $B_{j-1}$. We have that $u_1=u$. We define $$\varphi_j =\dfrac{{\rm cap}_q(B_j \setminus \Omega, B_{j-1})}{{\rm cap}_q (B_j, B_{j-1})}.$$
	Using Lemma 5.6 of \cite{BMS} and the standard inequality $1+t\leq e^t$, $t \in \mathbb{R}$, we get
	\begin{equation}\label{23BMS}
		u_j \geq C \varphi_j^{\frac{1}{q-1}} \geq 1-e^{-C \varphi_j^{\frac{1}{q-1}}}, \quad \mbox{in $\overline{B}_j$.}
	\end{equation}
Now, for $j \in \mathbb{N}$, we recursively define 
\begin{equation*}
v_j=\begin{cases}
1-u_1  \hspace{6.26cm} \mbox{if $j=1$,}\\
{\rm exp}\left(C \varphi_{j-1}^{\frac{1}{q-1}}v_{j-1}\right)= {\rm exp}\left(C\sum_{k=1}^{j-1} \varphi_k^{\frac{1}{q-1}}\right)v_1 \quad \mbox{if $j \geq2$}.
	\end{cases}
\end{equation*}
We prove now by induction that for all $j \in \mathbb{N}$, $v_j \leq {\rm exp}\Big(-C \varphi_{j}^{\frac{1}{q-1}}\Big)$ in $\overline{B}_j$.
	Firstly, by \eqref{23BMS}, we obtain $v_1 \leq {\rm exp}\Big(-C \varphi_{1}^{\frac{1}{q-1}}\Big)$ in $\overline{B}_1$. Assuming that $v_j \leq {\rm exp}\Big(-C \varphi_{j}^{\frac{1}{q-1}}\Big)$ in $\overline{B}_j$, then, by the definition of $v_{j+1}$, we get $v_{j+1}\leq 1$ in $\overline{B}_j$.  Since $v_1= 0$ in $\overline{B}_1\setminus \Omega \supset \overline{B}_j\setminus \Omega$, we then obtain that $v_{j+1}= 0$ in $\overline{B}_j\setminus \Omega$ and $1-u_{j+1}=1$ outside $B_j$. 
	We notice that $0 \leq u_{j+1}, v_{j+1} \leq 1$, thus $v_{j+1}\leq  1 - u_{j+1}$ on the boundary $\partial (B_j \cap \Omega)$. Now, $u_{j+1}$ and $v_{j+1}$ are $(p,q)$-minimizers in $B_j \cap \Omega$, as a consequence, the comparison principle Lemma \ref{Theorem 6.4Sh}
	and \eqref{23BMS} imply that $v_{j+1} \leq 1 - u_{j+1} \leq  {\rm exp}\Big(-C \varphi_{j+1}^{\frac{1}{q-1}}\Big)$ in $\overline{B}_{j+1}$. This concludes the induction step. From the definition of $v_j$, we get
	\begin{equation}\label{25BMS}		
		1-u=1-u_1=v_1\leq {\rm exp}\left(-C\sum_{k=1}^{j} \varphi_k^{\frac{1}{q-1}}\right) \quad \mbox{in $\overline{B}_{j}$.}
	\end{equation}
	We choose $\rho > 0$ so that $\rho \leq r$ and an integer $m$ satisfying $5^{-m}r < \rho \leq 5^{1-m}r$. 
	Again, from the definition of $\varphi_j$ and using Lemma 5.5 of \cite{BMS}, we get
	$$\varphi_j \geq C \dfrac{{\rm cap}_q (B(x_0,t) \setminus \Omega, B(x_0,2t))}{{\rm cap}_q(B(x_0,t), B(x_0,2t))}= C \varphi(x_0, X\setminus \Omega,t)$$
	if $5^{-j}r \leq t \leq 5^{1-j}r$. Thus,
	\begin{align*}
		\sum_{j=1}^{m} \varphi_j^{\frac{1}{p-1}} &\geq C \sum_{j=1}^{m} \int_{5^{-j}r}^{ 5^{1-j}r} \varphi(x_0, X\setminus \Omega,t)^{\frac{1}{q-1}} \frac{dt}{t}\\
		& \geq C  \int_{\rho}^{r} \varphi(x_0, X\setminus \Omega,t)^{\frac{1}{q-1}} \frac{dt}{t},
	\end{align*}
	and, by \eqref{25BMS}, we conclude the proof.
\end{proof}

Theorem \ref{Theorem 5.1BMS} now follows from Lemma \ref{Lemma5.7BMS} and the comparison principle
Lemma \ref{Theorem 6.4Sh} in exactly the same way as Theorem 6.18 in \cite{HKM} (see also \cite{BMS}).

\end{document}